\documentclass[12pt]{amsart}
\usepackage{amsthm,amsmath,amssymb,url,color,tikz}
\usepackage{graphicx}
\usetikzlibrary{patterns}
\usepackage[all]{xy}
\usepackage{young}
\usepackage{enumerate}
\usepackage[margin=1in]{geometry}
\usepackage[colorlinks=true,citecolor=black,linkcolor=black,urlcolor=blue]{hyperref}

\usetikzlibrary{arrows,automata}
\usepackage[colorinlistoftodos,color=red!70]{todonotes}

\theoremstyle{plain}
\newtheorem{thm}{Theorem}
\newtheorem{lem}[thm]{Lemma}
\newtheorem{prop}[thm]{Proposition}
\newtheorem{cor}[thm]{Corollary}

\theoremstyle{definition}
\newtheorem{defn}[thm]{Definition}
\newtheorem{ex}[thm]{Example}

\theoremstyle{remark}
\newtheorem{rem}[thm]{Remark}

\numberwithin{thm}{section}

\newcommand{\RR}{\mathbb{R}}
\newcommand{\ZZ}{\mathbb{Z}}
\newcommand{\QQ}{\mathbb{Q}}

\newcommand{\mcO}{\mathcal{O}}

\def\<{\langle}
\def\>{\rangle}
\def\longto{\longrightarrow}

\DeclareMathOperator\Acyc{Acyc}
\DeclareMathOperator\Conj{Conj}
\DeclareMathOperator\cl{cl}
\DeclareMathOperator\C{C}
\DeclareMathOperator\NC{NC}
\DeclareMathOperator\card{card}
\DeclareMathOperator\Togglable{Togglable}
\DeclareMathOperator\ind{Ind}

\begin{document}

\title{Noncrossing partitions, toggles, and homomesies}
\author[Einstein]{David Einstein}
\address{Department of Mathematics, University of Massachusetts Lowell, Lowell, MA, USA}
\email{deinst@gmail.com}
\author[Farber]{Miriam Farber}
\address{Department of Mathematics, Massachusetts Institute of Technology, Cambridge, MA 02139, USA}
\email{mfarber@mit.edu}
\author[Gunawan]{Emily Gunawan}
\address{School of Mathematics, University of Minnesota, Minneapolis, MN 55455, USA}
\email{egunawan@umn.edu}
\author[Joseph]{Michael Joseph}
\address{Department of Mathematics, University of Connecticut, Storrs, CT 06269-1009, USA}
\email{michael.j.joseph@uconn.edu}
\author[Macauley]{Matthew Macauley}
\address{Department of Mathematical Sciences, Clemson University, Clemson, SC 29634-0975, USA}
\email{macaule@clemson.edu}
\author[Propp]{James Propp}
\address{Department of Mathematics, University of Massachusetts Lowell, Lowell, MA, USA}
\email{{\rm See} http://jamespropp.org}
\author[Rubinstein-Salzedo]{Simon Rubinstein-Salzedo}
\address{Euler Circle, Palo Alto, CA 94306, USA}
\email{simon@eulercircle.com}
\date{\today}


\maketitle

%\subjclass[2010]{05E18}
%\keywords{Coxeter element, homomesy, involution, noncrossing partition, toggle group}

\begin{abstract}
We introduce $n(n-1)/2$ natural involutions (``toggles'') on the set
$S$ of noncrossing partitions $\pi$ of size $n$, along with certain
composite operations obtained by composing these involutions.  We show
that for many operations $T$ of this kind, a surprisingly large family
of functions $f$ on $S$ (including the function that sends $\pi$ to
the number of blocks of $\pi$) exhibits the homomesy phenomenon: the
average of $f$ over the elements of a $T$-orbit is the same for all
$T$-orbits.  We can apply our method of proof more broadly to toggle
operations back on the collection of independent sets of certain
graphs. We utilize this generalization to prove a theorem about toggling on a family of graphs called ``$2$-cliquish.'' More generally, the philosophy of this ``toggle-action,'' proposed by Striker, is a popular topic of current and future research in dynamic algebraic combinatorics.

\end{abstract}

%\begin{abstract}
%We introduce $n(n-1)/2$ natural involutions (``toggles'') on the set $S$ of noncrossing partitions $\pi$ of size $n$, along with certain composite operations obtained by composing these involutions. We show that for many operations $T$ of this kind, a surprisingly large family of functions $f$ on $S$ (including the function that sends $\pi$ to the number of blocks of $\pi$) exhibits the homomesy phenomenon: the average of $f$ over the elements of a $T$-orbit is the same for all $T$-orbits. Our method of proof applies more broadly to toggle operations on the collection of independent sets of certain graphs.
%\end{abstract}

%%=======================================
\section{Introduction}
%%=======================================

 A {\em partition} of $[n]:=\{1,2,\dots,n\}$ is a collection $\pi$ of disjoint sets $B_1,B_2,\dots,B_K$ with union $[n]$. We call the $B_i$'s ``blocks'' and write $|\pi|=K$.
A partition $\pi$ is {\em noncrossing} if whenever $1 \leq i < j < k < \ell \leq n$,
we do not have $i$ and $k$ belonging to one block of $\pi$
with $j$ and $\ell$ belonging to a different block.
(For motivation of the term ``noncrossing'',
see the discussion of the linear representation of $\pi$ below, and the circular representation in Section~\ref{sec:toggling}.)
Simion and Ullman~\cite{simion91structure}
define an involution $\lambda$ on the set of noncrossing partitions of $[n]$
(they call it $\alpha$)
with the property that $|\pi| + |\lambda(\pi)| = n+1$.
This map is related to
a different operation on noncrossing partitions,
the \emph{Kreweras complementation} \cite{kreweras1972sur}, denoted $\kappa$.
The bijection $\kappa$ is not an involution
but it too satisfies $|\pi| + |\kappa(\pi)| = n+1$.
The actions $\kappa$ and $\lambda$ have very different orbit-structures,
but they share the property that the average of $|\pi|$ over each orbit is $(n+1)/2$.
That is, in the terminology of Propp and Roby ~\cite{propp2015homomesy}, the statistic $\pi \mapsto |\pi|$
is {\em homomesic} under the actions of $\kappa$ and $\lambda$,
or more specifically, $c$-mesic with $c=(n+1)/2$.

In this paper, we exhibit a large class of actions
sharing this homomesy property with $\kappa$ and $\lambda$.
These actions are obtained from the {\em toggle-action}
philosophy first studied by Cameron and Fon-Der-Flaass~\cite{cameron1995orbits} and more recently by Striker and Williams~\cite{striker2012promotion} and Striker~\cite{striker2016rowmotion}.
This philosophy invites us to act on combinatorial objects
via operations obtained as compositions of
extremely simple involutions.
These involutions (called {\em toggles}) may have many fixed points;
 indeed, it is usually the case that for any given toggle operation $\tau$
acting on a set $S$ of combinatorial objects being studied,
most of the elements of $S$ are fixed by $\tau$.
However, by composing many toggles
we obtain a permutation $T$ of $S$ that mixes $S$ up
more than any individual toggle does.
Propp and Roby~\cite{propp2015homomesy} add to this picture
the observation that in many cases of interest,
$T$ does such a good job of mixing up $S$
that, for some interesting numerical statistics $f$ on $S$,
the average of $f$ on each $T$-orbit is some constant
that only depends on $f$, not what orbit we are in.

In this article,
$S$ is the set of noncrossing partitions $\pi$ of $[n]$
and $f(\pi)$ is $|\pi|$ or various related quantities.
To define the sorts of toggles we use,
we make use of the {\em linear representation} of noncrossing partitions, as shown in Figure~\ref{nclinear}.
This representation of $\pi$
depicts the numbers $1,\dots,n$ as equally-spaced points on a horizontal line
and consists of arcs above the line joining points $i$ and $j$
whenever $i$ and $j$ are successive elements of the same block.
Formally, the linear representation $P$ of $\pi$
consists of those pairs $(i,j)$ with $1 \leq i < j \leq n$
with the property that $i$ and $j$ are in the same block of $\pi$
but none of $i+1,i+2,\dots,j-1$ (the ``interior'' of the arc $(i,j)$) are also in that block.
The noncrossing property of the partition guarantees that
if two arcs belong to $P$,
then their interiors are disjoint,
their left endpoints are distinct,
and their right endpoints are distinct.
(That is, we never see two arcs exhibiting
the three forbidden configurations
depicted in Figure~\ref{fig:noncommuting-toggles},
called respectively {\em crossing},
{\em left-nesting}, and {\em right-nesting}.
Note however that nested arcs are allowed.)
Conversely, any collection of arcs
satisfying these conditions
determines a unique noncrossing partition $\pi$.
 
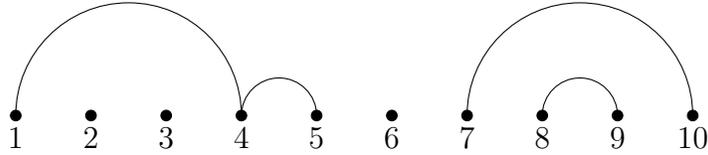
\begin{figure}
\begin{center}
\begin{tikzpicture}
\foreach \x in {1,...,10}
{
\filldraw (\x,0) circle (2pt);
\path node at (\x,-.3) {\x};
}
\draw (1,0) arc (180:0:1.5cm);
\draw (4,0) arc (180:0:0.5cm);
\draw (7,0) arc (180:0:1.5cm);
\draw (8,0) arc (180:0:0.5cm);
\end{tikzpicture}
\end{center}
\caption{The linear representation
$P=\{(1,4),(4,5),(7,10),(8,9)\}$
of the noncrossing partition 
$\pi=\{\{1,4,5\},\{2\},\{3\},\{6\},\{7,10\},\{8,9\}\}$.}
\label{nclinear}
\end{figure}

\begin{figure}
\begin{center}
\begin{tikzpicture}
\begin{scope}[shift={(0,0)}]
\filldraw (0,0) circle (2pt);
\filldraw (1,0) circle (2pt);
\filldraw (2,0) circle (2pt);
\filldraw (3,0) circle (2pt);
\draw (0,0) arc (180:0:1);
\draw (1,0) arc (180:0:1);
\end{scope}
\begin{scope}[shift={(5,0)}]
\filldraw (0,0) circle (2pt);
\filldraw (1,0) circle (2pt);
\filldraw (2,0) circle (2pt);
\draw (0,0) arc (180:0:1);
\draw (0,0) arc (180:0:.5);
\end{scope}
\begin{scope}[shift={(9,0)}]
\filldraw (0,0) circle (2pt);
\filldraw (1,0) circle (2pt);
\filldraw (2,0) circle (2pt);
\draw (0,0) arc (180:0:1);
\draw (1,0) arc (180:0:.5);
\end{scope}
\end{tikzpicture}
\end{center}
\caption{Disallowed pairs of arcs in a noncrossing partition: crossing, left-nesting, and right-nesting, respectively.}
\label{fig:noncommuting-toggles}
\end{figure}
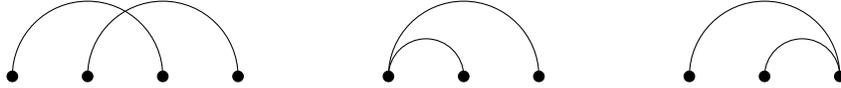

For each pair $i,j$ with $1 \leq i < j \leq n$,
the toggle operation $\tau_{i,j}$ can be summarized as follows: ``\emph{If arc $(i,j)$ is present, then remove it. If it is missing, add it if possible.}'' More formally: Given a noncrossing partition $\pi$ of $[n]$, draw its linear representation $L$,
which either (1) contains the arc $(i,j)$ or (2) does not.
In case (1), let $L'$ be $L$ with arc $(i,j)$ removed;
in case (2), let $L'$ be $L$ with arc $(i,j)$ added.
If $L'$ is the linear representation of some noncrossing partition $\pi'$
(guaranteed to exist in case (1) but not guaranteed to exist in case (2)),
let $\tau_{i,j}(\pi) = \pi'$; 
otherwise let $\tau_{i,j}(\pi) = \pi$.
For example, when $\pi$ is the noncrossing partition
whose linear representation appears in Figure~\ref{nclinear},
applying $\tau_{i,j}$ removes the edge $(i,j)$
when $(i,j)$ is $(1,4)$, $(4,5)$, $(7,10)$, or $(8,9)$,
and adds the edge $(i,j)$
when $(i,j)$ is $(2,3)$, $(5,6)$, $(5,7)$, or $(6,7)$;
for all other pairs $(i,j)$, $\tau_{i,j}$ has no effect on $\pi$.

Our main result (Theorem~\ref{thm:arccounthomomesy})
is that $\pi \mapsto |\pi|$ is $(n+1)/2$-mesic
under a very large class of operations $T$ obtained
as compositions of toggles
(even though for most of our operations $T$
it is not the case that $|\pi|+|T(\pi)|=n+1$ for all $\pi$).
Careful definitions and statements of theorems are given in the next section;
succeeding sections provide proofs
and discussion of side-issues.
% We might want to say more about the structure of the article.

Since we will be dealing almost exclusively
with noncrossing partitions by way of their linear representations,
we will in many parts of this article abuse terminology
by referring to these linear representations as noncrossing partitions.
When we have occasion to refer directly to the blocks $B_1,\dots,B_K$
rather than to the arcs,
we will call $\pi = \{B_1,\dots,B_K\}$
the {\em block representation} of the noncrossing partition $P$,
where $P$ is a set of arcs.

\section*{Acknowledgments}

This article was made possible by the
American Institute of Mathematics,
which sponsored the workshop on
Dynamical Algebraic Combinatorics
organized by Propp, Roby, Striker, and Williams;
the seven authors' success in proving
and generalizing Propp's original conjecture
(presented on the first day of the workshop)
has made all seven keenly appreciative
of the wisdom of the AIM workshop format
and its emphasis on collaborative brainstorming
rather than traditional lecturing.
The authors met for several hours 
during each afternoon of the workshop.
Progress during those five afternoons was dramatic
but in a sense ``negative'';
instead of proving our conjectures
we only succeeded in finding more of them.
We continued to email one another during the weeks that followed,
and the ideas behind the proof only materialized
a month later. We would also like to thank the referee for useful suggestions.

%%===========================================================
 \section{Toggling noncrossing partitions}\label{sec:toggling}
%%===========================================================

In addition to employing the linear representation
of noncrossing partitions via arcs
(for purposes of defining the toggle operations), 
% (for purposes of defining not only the toggle operations but also some of the statistics on noncrossing partitions that will turn out to be homomesic under suitable compositions mof those toggles)
we will have occasion to use
the more classical {\em circular representation} 
of noncrossing partitions (for purposes of defining the Kreweras complement 
and clarifying its relation to the Simion-Ullman involution).
This representation of $\pi$
depicts the numbers $1,\dots,n$ as equally-spaced points on a circle (by convention arranged clockwise)
and the blocks as convex hulls.
Figure~\ref{ncgraphics} shows the linear and circular
representations of the noncrossing partition
$\pi=\{\{1\},\{2,4,5\},\{3\},\{6,8\},\{7\}\}$. The noncrossing property ensures that the convex hulls are pairwise disjoint, i.e., the blocks are ``noncrossing.''
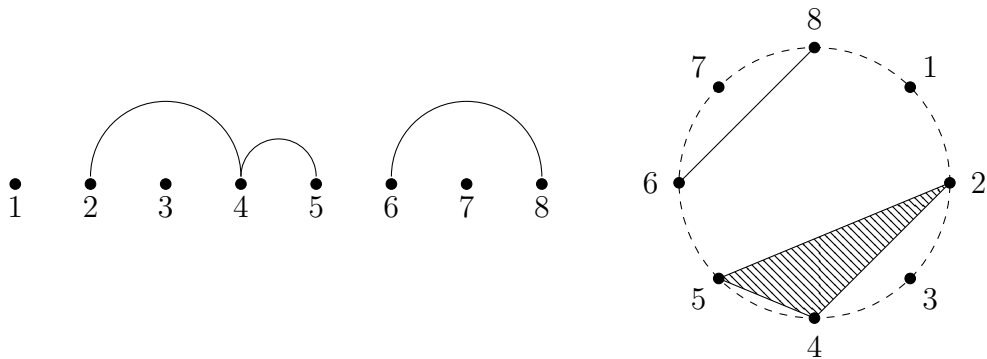
\begin{figure}
\begin{center}
\begin{tikzpicture}
\foreach \x in {1,...,8}
{
\filldraw (\x,-1.1) circle (2pt);
\path node at (\x,-1.4) {\x};
}
\draw (2,-1) arc (180:0:1cm);
\draw (4,-1) arc (180:0:.5cm);
\draw (6,-1) arc (180:0:1cm);
\path node at (4.5,-3.4) { };
\end{tikzpicture}
\qquad
\begin{tikzpicture}
\draw[dashed] (0,0) circle (1.8cm);
\foreach \x in {1,...,8}
{
\filldraw (90-45*\x:1.8) circle (.07cm);
\path node at (90-45*\x:2.18) {\x};
}
\draw[pattern = north west lines] (0:1.8) -- (-90:1.8) -- (-135:1.8) -- cycle;
\draw (180:1.8) -- (90:1.8);
\end{tikzpicture}
\end{center}
\caption{The linear and circular representations of
the noncrossing partition 
$\pi=\{\{1\},\{2,4,5\},\{3\},\{6,8\},\{7\}\}$.}
\label{ncgraphics}
\end{figure}

Let $\NC(n)$ denote the set of noncrossing partitions of $[n]$. As was mentioned in the introduction, we sometimes consider elements of $\NC(n)$ as sets of arcs such that the corresponding ``arc diagram'' is free of
the disallowed configurations shown in Figure~\ref{fig:noncommuting-toggles}, and we sometimes instead speak of a noncrossing partition as having blocks, and will refer to the associated collection $\{B_1,\dots,B_K\}$ as the block representation of $\pi$. We shall typically use uppercase Roman letters, especially $P$, to refer to the arc representation and lowercase Greek letters, especially $\pi$, to refer to the block representation. 

For a fixed $[n]$, there is a natural partial order on the set of noncrossing partitions by \emph{refinement}: $\pi\leq\pi'$ if each block of $\pi$ is contained in a block of $\pi'$. This endows $\NC(n)$ with a lattice structure. More details on this and other properties of $\NC(n)$ can be found in the fun survey article \cite{mccammond2006noncrossing}. Note that removing an arc from a nonempty noncrossing partition $P$ yields another noncrossing partition strictly finer than $P$, but this ``subset order'' is \emph{not} the same as the refinement order. For example, the noncrossing partition in $\NC(3)$ consisting of the single arc $(1,3)$ has two blocks. It is a refinement of, but not a subset of, the noncrossing partition consisting of the arcs $(1,2)$ and $(2,3)$, which has only one block. 

Recall that we have informally defined 
$\tau_{i,j}$ (for $1\leq i<j\leq n$)
to be the permutation of $\NC(n)$ defined by adding/removing the arc $(i,j)$ to/from each noncrossing partition whenever possible, and doing nothing otherwise. 
The formal definition of the toggle operations follow.

\begin{defn} 
Given a pair $(i,j)$ with $1\le i<j\le n$, the \emph{toggle operation} $\tau_{i,j}$ on $\NC(n)$ is defined to be 
\[
\tau_{i,j}(P)=\begin{cases} P\cup\{(i,j)\} & (i,j)\not\in P \text{ and } P\cup\{(i,j)\}\in \NC(n), \\ P\setminus\{(i,j)\} & (i,j)\in P, \\ P & \text{otherwise.} \end{cases}
\] 
The \emph{toggle group} $W_n$ is the subgroup of the permutation group $S_{\NC(n)}$ generated by the $\binom{n}{2}$ toggle operations. We write toggles from right-to-left, so $\tau_{i,j}\tau_{k,\ell}:=\tau_{i,j}\circ\tau_{k,\ell}$.
\label{def:toggle_NC(n)}
\end{defn}

It is clear that each toggle operation is an involution. The object of this paper is to understand some well-behaved statistics of the toggle group and its action on $\NC(n)$. We define the relevant notions in Section~\ref{sec:homomesy}. The choice of $W_n$ to denote the  toggle group is motivated by the fact that it is always a quotient of a Coxeter group, which is classically denoted by $W$. We will revisit this in Section~\ref{sec:coxeter}.

It will be helpful for what follows to classify which pairs of toggles do and do not commute. Any pair of distinct arcs $(i,j)$ and $(k,\ell)$ can be classified into one of six types (possibly after swapping $(i,j)$ with $(k,\ell)$): 
\begin{enumerate} 
\item $i<j<k<\ell$ (disjoint), 
\item $i<k<\ell<j$ (nesting), 
\item $i<j=k<\ell$ ($m$-shaped), 
\item $i=k<j<\ell$ (left-nesting), 
\item $i<k<j=\ell$ (right-nesting), 
\item $i<k<j<\ell$ (crossing). \end{enumerate} 
The type is sufficient to determine whether or not the pair of toggles commutes.

%, and it turns out that half of these correspond to commuting toggles:
%%
%\begin{itemize}
%\item $\tau_{i,j}$ and $\tau_{k,\ell}$ do not commute if $\{i,j\}$ and $\{k,\ell\}$ are either crossing or half-nesting, i.e., one of the three examples shown in Figure~\ref{fig:noncommuting-toggles}.
%\item $\tau_{i,j}$ and $\tau_{k,\ell}$ commute if $\{i,j\}$ and $\{k,\ell\}$ are either disjoint, nesting, or m-shaped, i.e., one of the three examples shown in Figure~\ref{fig:commuting-toggles}.
%\end{itemize}

%The formal statement of this claim is given in Proposition~\ref{prop:commuting-toggles}.

\begin{prop} \label{prop:commuting-toggles} 
Let $\tau_{i,j}$ and $\tau_{k,\ell}$ be distinct toggles. Then $\tau_{i,j}$ and $\tau_{k,\ell}$ commute if and only if the arcs $(i,j)$ and $(k,\ell)$ are disjoint, nesting, or $m$-shaped.
\end{prop}

\begin{proof} Suppose first that $(i,j)$ and $(k,\ell)$ are left-nesting, right-nesting, or crossing. Let $P=\{\}$ be the empty noncrossing partition of $[n]$. Then $\tau_{i,j}\tau_{k,\ell}(\pi)=\{(k,\ell)\}$, whereas $\tau_{k,\ell}\tau_{i,j}(\pi)=\{(i,j)\}$. Thus $\tau_{i,j}$ and $\tau_{k,\ell}$ do not commute.

On the other hand, suppose that $(i,j)$ and $(k,\ell)$ are disjoint, nested, or $m$-shaped. Then adding or removing $(i,j)$ does not interfere with adding or removing $(k,\ell)$, so $\tau_{i,j}$ and $\tau_{k,\ell}$ commute.
\end{proof}

The commuting pairs are illustrated in Figure~\ref{fig:commuting-toggles}, and the non-commuting pairs were shown back in Figure~\ref{fig:noncommuting-toggles}.

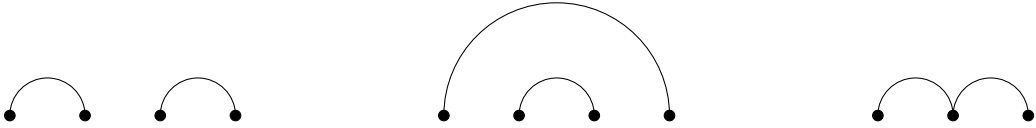
\begin{figure}
\begin{center}
\begin{tikzpicture}
\filldraw (0,0) circle (2pt);
\filldraw (1,0) circle (2pt);
\filldraw (2,0) circle (2pt);
\filldraw (3,0) circle (2pt);
\draw (0,0) arc (180:0:.5);
\draw (2,0) arc (180:0:.5);
\end{tikzpicture}
\qquad\qquad\qquad
\begin{tikzpicture}
\filldraw (0,0) circle (2pt);
\filldraw (1,0) circle (2pt);
\filldraw (2,0) circle (2pt);
\filldraw (3,0) circle (2pt);
\draw (0,0) arc (180:0:1.5);
\draw (1,0) arc (180:0:.5);
\end{tikzpicture}
\qquad\qquad\qquad
\begin{tikzpicture}
\filldraw (0,0) circle (2pt);
\filldraw (1,0) circle (2pt);
\filldraw (2,0) circle (2pt);
\draw (0,0) arc (180:0:.5);
\draw (1,0) arc (180:0:.5);
\end{tikzpicture}
\end{center}
\caption{Commuting pairs of toggles: disjoint, nesting, and $m$-shaped, respectively.}
\label{fig:commuting-toggles}
\end{figure}

\begin{cor} \label{do not commute} Given $i<j$ with $j-i=m$, there are $m(n+1-m)-2$ toggles $\tau_{k,\ell}$ that do not commute with $\tau_{i,j}$.
\end{cor}

\begin{proof}
Let $i<j$ with $j-i=m$. We will classify the toggles that do not commute with $\tau_{i,j}$.

If $\ell>j$, then $\tau_{i,\ell}$ does not commute with $\tau_{i,j}$, and if $k<i$, then $\tau_{k,j}$ does not commute with $\tau_{i,j}$.  This type of toggle is in one-to-one correspondence with the numbers in $[n]$ that are less than $i$ or greater than $j$, and there are $n-m-1$ such numbers.

The other way that $\tau_{k,\ell}$ will not commute with $\tau_{i,j}$ is if one of $k$ or $\ell$ is strictly between $i$ and $j$, and the other is not strictly between $i$ and $j$.  There are $m-1$ numbers in $[n]$ that are strictly between $i$ and $j$, and the other $n+1-m$ numbers in $[n]$ are not strictly between $i$ and $j$, so there are $(m-1)(n+1-m)$ pairs $k,l$ of this type.

Adding up the two cases, there are $n-m-1+(m-1)(n+1-m)=m(n+1-m)-2$ toggles $\tau_{k,\ell}$ that do not commute with $\tau_{i,j}$.
\end{proof}

Proposition \ref{prop:commuting-toggles}, together with the following result, shows how the order of the product of two toggles is determined only by their ``type.'' This should motivate the connection to Coxeter theory, which will be described in more detail in the following section. 

\begin{prop}\label{prop:m}
For any pair of toggles $\tau_{i,j}$ and $\tau_{k,\ell}$, let $m(\tau_{i,j}\tau_{k,\ell})$ denote the order of the element $\tau_{i,j}\tau_{k,\ell}$ in $W_n$. Then  
\[
m(\tau_{i,j},\tau_{k,\ell})=\left\{
\begin{array}{ll}
      1 & \text{if } (i,j)=(k,\ell),\\
      2 & \text{if } \tau_{i,j},\tau_{k,\ell} \text{ commute and }(i,j)\not=(k,\ell),\\
      6 & \text{if } \tau_{i,j},\tau_{k,\ell} \text{ do not commute.}
\end{array} 
\right.
\]
\end{prop}

\begin{proof}
It is clear from the fact that each toggle is an involution and no two toggle operations are the same that $m(\tau_{i,j},\tau_{k,\ell})=2$ if $\tau_{i,j}$ and $\tau_{k,\ell}$ commute unless $(i,j)=(k,\ell)$ in which case $m(\tau_{i,j},\tau_{k,\ell})=1$. 

Now suppose $\tau_{i,j}$ and $\tau_{k,\ell}$ do not commute. By Proposition~\ref{prop:commuting-toggles}, the arcs are either left-nesting, right-nesting, or $m$-shaped, and so no noncrossing partition can contain both $(i,j)$ and $(k,\ell)$.  Since the maps $\tau_{i,j}$ and $\tau_{k,\ell}$ only affect two arcs in a noncrossing partition, there can be at most three noncrossing partitions in any orbit of $\tau_{i,j}\tau_{k,\ell}$.

Applying $\tau_{i,j}\tau_{k,\ell}$ to the empty partition $\{\}$ gives $\{(k,\ell)\}$.  Applying $\tau_{i,j}\tau_{k,\ell}$ to $\{(k,\ell)\}$ gives $\{(i,j)\}$, and then applying $\tau_{i,j}\tau_{k,\ell}$ again gives $\{\}$, so this is an orbit of size 3.

Let $A$ be any arc that can be in the same noncrossing partition as one of $(i,j)$ and $(k,\ell)$, but not the other. Figure \ref{fig:noncommuting-toggles} shows that we can always find such an arc. Specifically, if these arcs are crossing with $i<k<j<\ell$, or left-nesting with $i=k<j<\ell$, then $A=(j,\ell)$ will work. The right-nesting case is analogous. Without loss of generality, assume $A$ can be in the same noncrossing partition as $(i,j)$ but not $(k,\ell)$.  Then there is an orbit of size 2 containing the noncrossing partitions $\{A\}$ and $\{A,(i,j)\}$. 

As there exists an orbit of size 2 and an orbit of size 3, and no orbit has size larger than 3, $\tau_{i,j}\tau_{k,\ell}$ has order 6.
\end{proof}

Let $C_n$ denote the $n^\text{th}$ Catalan number, where $C_0=C_1=1$. It is well-known that the cardinality $|\NC(n)|$ is $C_n$ \cite{mccammond2006noncrossing}. The enumeration of certain subsets of $\NC(n)$ based on toggles can also be expressed in terms of Catalan numbers, as in the following proposition.

\begin{prop}\label{prop:numberoffixedpartitions_depend_on_j_minus_i}
Define $\NC(n)_{i,j}$ to be the set of noncrossing partitions containing the arc $(i,j)$ 
% (in our Genji notation),
and $\Togglable(n)_{i,j} := \{ P\in \NC(n)\mid (i,j)\notin P$ but $(i,j)\in\tau_{i,j}(P) \}$.
\begin{enumerate}
\item $|\Togglable(n)_{i,j}| = |\NC(n)_{i,j}|$.
\item \label{item:lem:numberfixedbytoggle_ij}
$|\NC(n)_{i,i+k}|=C_{n-k}C_{k-1}$. In particular, $|\NC(n)_{i,i+1}|=C_{n-1}$.
\item The number of partitions $\pi\in\NC(n)$ fixed by $\tau_{i,i+k}$ is $$C_n-2|\NC(n)_{i,i+k}|=C_n-2C_{n-k}C_{k-1}.$$
\item $|\NC(n)_{i,i+k}|=|\NC(n)_{i,i+n+1-k}|$.
\end{enumerate}
\end{prop}

\begin{proof}
\begin{enumerate}
\item The toggle $\tau_{i,j}$ gives a bijection between $\NC(n)_{i,j}$ and $\Togglable(n)_{i,j}$.
\item The details are omitted, but it follows from the standard recurrence for showing that $|\NC(n)|=C_n$.%Since $\NC(n)_{i,i+k}$ is the set of partitions in $\NC(n)$ containing $(i,i+k)$, it is the set of partitions for which $i$ and $i+k$ are in the same block, and none of the elements $i+1,\dots,i+k-1$ are in the same block as $i$ and $i+k$.  We construct a bijection between $\NC(n)$ containing $(i,i+k)$ and the set of pairs $(Q,R)$ such that $Q$ is a noncrossing partition of $\{1,2,\ldots,i-1,i+k,i+k+1,\ldots,n\}$ and $R$ is a noncrossing partition of $\{i+1,i+2,\ldots,i+k-1\}$.

%Let $P\in\NC(n)$ such that $(i,i+k)\in P$. Then we define $Q$ to be the noncrossing partition of $S_{i,k}:=\{1,2,\ldots,i-1,i+k,i+k+1,\ldots,n\}$ where two elements of $S_{i,k}$ are in the same block in $Q$ if and only if they are in the same block in $P$.  Similarly define $R$ to be the noncrossing partition of $T_{i,k}:=\{i+1,\ldots,i+k-1\}$ where two elements of $T_{i,k}$ are in the same block in $T$ if and only if they are in the same block in $P$.  Notice that we are describing $Q$ and $R$ using the block representation, not the arc representation.

%For the reverse bijection, given $Q$ and $R$, we let $P\in\NC(n)$ be the partition where two elements of $[n]\setminus\{i\}$ are in the same block of $P$ if and only if they are in the same block in either $Q$ or $R$, and also place $i$ in the block with $i+k$. It is clear that these maps are inverses of one another.  Clearly, the number of such pairs $(Q,R)$ is equal to $C_{n-k}C_{k-1}$. So the conclusion follows.
\item The conclusion follows from the above two items.
\item By part~(\ref{item:lem:numberfixedbytoggle_ij}), we have $\NC(n)_{i,i+(n+1-k)} = C_{n-(n+1-k)}C_{n+1-k-1} = C_{k-1} C_{n-k}$.
\end{enumerate}
\end{proof}

The commutation relations between the toggle operations 
can be described by a undirected graph, called the \emph{base graph}.

\begin{defn}\label{defn:base-graph}
The base graph $\Gamma_n$ of $W_n$ is the graph $(V,E)$, where $V=\{\tau_{i,j}\mid i<j\}$, and $E$ consists of edges of the form $\{\tau_{i,j},\tau_{k,\ell}\}$, where $\tau_{i,j}$ and $\tau_{k,\ell}$ are non-commuting toggles. \end{defn}

It is easiest to arrange the vertex set $\{\tau_{i,j}\mid 1\leq i<j\leq n\}$ in an upper-triangular grid, as shown in Figure~\ref{fig:base-graph}. Each row is a clique (complete subgraph), as is each column; these correspond to the half-nesting pairs. Finally, there are some ``diagonal edges,'' which correspond to crossing pairs: $(i,j)$ and $(k,\ell)$ where $i<k<j<\ell$. All of these are ``negatively sloped'' in the sense of calculus. There are no ``positively sloped'' diagonal edges.

\begin{figure}
\begin{center}
\begin{tikzpicture}[scale=1.5,auto,shorten >=-2pt,shorten <=-2pt]
  \node (12) at (0,3) {{\tiny$\tau_{1,2}$}};
  \node (13) at (1,3) {{\tiny$\tau_{1,3}$}};
  \node (14) at (2,3) {{\tiny$\tau_{1,4}$}};
  \node (15) at (3,3) {$\dots$};
  \node (1n) at (4,3) {{\tiny$\tau_{1,n}$}};
  \node (23) at (1,2) {{\tiny$\tau_{2,3}$}};
  \node (24) at (2,2) {{\tiny$\tau_{2,4}$}};
  \node (25) at (3,2) {$\dots$};
  \node (2n) at (4,2) {{\tiny$\tau_{2,n}$}};
  \node (34) at (2,1) {{\tiny$\tau_{3,4}$}};
  \node (35) at (3,1) {$\dots$};
  \node (3n) at (4,1) {{\tiny$\tau_{3,n}$}};
  \node at (3,.25) {$\ddots$};
  \node at (4,.5) {$\vdots$};
  \node (n-2n) at (4,0) {{\tiny$\tau_{n-2,n}$}};
  \node (n-1n) at (4,-1) {{\tiny$\tau_{n-1,n}$}};
  \draw (12) to (13);\draw (13) to (14);
  \draw (13) to (2n);\draw (14) to (2n);\draw (14) to (3n);
  \draw (12) to [bend left=20] (14);
  \draw (12) to [bend left=20] (1n);
  \draw (13) to [bend left=20] (1n);
  \draw (14) to [bend left=20] (1n);
  \draw (23) to (24);\draw (24) to (3n);
  \draw (23) to [bend left=20] (2n);
  \draw (24) to [bend left=20] (2n);
  \draw (34) to [bend left=20] (3n);
  \draw (1n) to (2n);\draw (2n) to (3n);
  \draw (1n) to [bend left=25] (3n);
  \draw (2n) to [bend left=25] (n-1n);
  \draw (3n) to [bend left=25] (n-1n);
  \draw (1n) to [bend left=25] (n-1n);
  \draw (2n) to [bend left=25] (n-2n);
  \draw (3n) to [bend left=25] (n-2n);
  \draw (1n) to [bend left=25] (n-2n);
  \draw (14) to (24);\draw (24) to (34);
  \draw (14) to [bend left=25] (34);
  \draw (13) to (23);
  \draw (13) to (24);\draw (24) to (35);
  \draw (3.95,0) to (3.5,.5);\draw (n-2n) to (n-1n);
\end{tikzpicture}
\end{center}
\caption{The base-graph $\Gamma_n$ of the toggle group $W_n$.}
\label{fig:base-graph}
\end{figure}
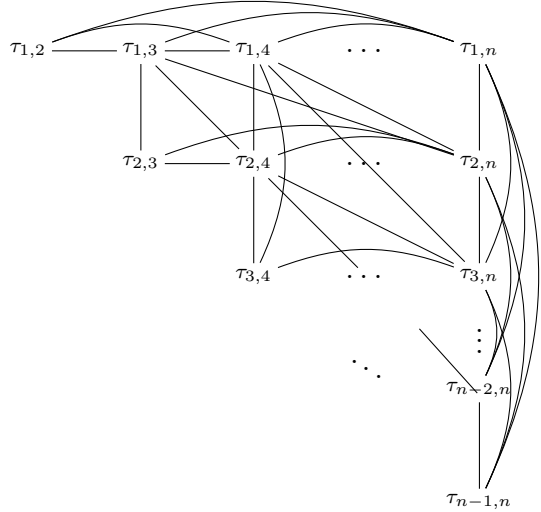

In summary, the base graph $\Gamma_n$, when drawn as in Figure~\ref{fig:base-graph}, has three types of edges: horizontal, vertical, and diagonal. Thus, it is possible to describe certain acyclic orientations of $\Gamma_n$ as e.g.\ ``orienting all edges east, south, and southeast'' or ``orienting all edges east, north, and southeast.''

%%=========================================
\section{Coxeter groups}\label{sec:coxeter}
%%=========================================

Since the toggle group is generated by involutions, it is a quotient of a Coxeter group \cite{bjorner2005combinatorics}. 

\begin{defn}
A \emph{Coxeter system} of rank $r$ is a pair $(W,S)$ consisting of a group $W$ generated by a set $S=\{s_1,\dots,s_r\}$ of involutions with presentation  
\[
W=\<s_1,\dots,s_r\mid s_i^2=(s_i s_j)^{m_{ij}}=1\>,
\]
where $m(s_i,s_j):=m_{ij}\geq 2$ for $i\neq j$. A \emph{reduced expression} of an element $w\in W$ is an expression $w=s_{x_1}s_{x_2}\cdots s_{x_\ell}$ such that $\ell$ is minimal, and $\ell$ is called the \emph{length} of $w$. The \emph{Coxeter graph} $\Gamma$ of $(W,S)$ is the undirected graph with vertex set $S$ and undirected edges $\{s_i,s_j\}$ for each $m_{ij}>2$. Edges are labeled with $m_{ij}$, though labels of $3$ are usually suppressed because they are the most common in Coxeter theory. It is also possible for $m_{ij}$ to be infinity, meaning $s_is_j$ has infinite order, but this is never the case for the Coxeter groups we will work with in this paper. A \emph{Coxeter element} of $W$ is a product $\prod_{i=1}^r s_{\sigma(i)}$ for some permutation $\sigma\in S_r$, i.e.\ a product of all the generators, each used exactly once, in some order. The set of Coxeter elements is denoted $\C(W,S)$, or $\C(W)$ if $(W,S)$ is understood. 
\end{defn}

\begin{rem}\label{rem:universal}
  Fix an ordering $\{\tau_1,\dots,\tau_r\}$ of the $r=\binom{n}{2}$ generators of the toggle group $W_n$. There is a canonical quotient $W\to W_n$, sending $s_i\mapsto\tau_i$, where $W$ is the Coxeter group of rank $r$ whose Coxeter graph is the base-graph $\Gamma_n$ with all edge weights $6$. This is from Proposition \ref{prop:m}. 
  
 Though $W_n$ can be realized as the quotient of many Coxeter groups, $W$ is in some sense ``minimal'' in that it satisfies the following universal property (stated without proof): \emph{For the toggle group $W_n=\<\tau_1,\dots\tau_n\>$, the Coxeter group $W=\<s_1,\dots,s_n\>$ and quotient $f\colon s_i\mapsto\tau_i$ has the property that for any other Coxeter group $W'=\<s'_1,\dots,s'_n\>$ and quotient $g:s'_i\mapsto\tau_i$, there is a unique homomorphism $h:W'\to W$ such that $h\circ f=g$.}
\end{rem}

%Note that the base-graph $\Gamma_n$ of $W_n$, from Definition \ref{defn:base-graph} and Figure \ref{fig:base-graph}, is simply the unlabeled version of the Coxeter graph $\Gamma$. Since all edge-labels of $\Gamma$ are $6$, no information is lost by henceforth considering $\Gamma_n$ instead of $\Gamma$.

%Let $\Gamma_n$ be an undirected graph with vertex set $V=\{\tau_{i,j}:i<j\}$ and edges $\{\tau_{i,j},\tau_{k,l}\}$ for non-commuting pairs. See Lemma \ref{conjtoggles}.

%Though two different Coxeter systems $(W,S)$ and $(W',S')$ may give rise to isomorphic groups $W\cong W'$, when we speak of a group $W$, we assume that the generating set $S$ is understood. 

If $\omega$ is an acyclic orientation of a graph $\Gamma$, then the pair $(\Gamma,\omega)$ defines a canonical partial order $P(\Gamma,\omega)$ on the vertex set, where $i<_{P(\Gamma,\omega)}j$ if there is an $\omega$-directed path from $i$ to $j$. If $\Gamma$ is understood, then we denote this partial order by $P_\omega:=P(\Gamma,\omega)$, and say it is a \emph{poset over $\Gamma$}. Each Coxeter element $c\in\C(W,S)$ defines an acyclic orientation $\omega(c)$ of the Coxeter graph $\Gamma$: orient the edge $\{s_i,s_j\}$ as $s_i\to s_j$ iff $s_i$ appears before $s_j$ in $c$. Shi in~\cite[Proposition 1.3]{shi1997enumeration} shows that this acyclic orientation is well-defined, i.e.\ it depends only on $c$, rather than on a choice of reduced expression. Thus, each Coxeter element $c\in\C(W)$ defines a poset $P_{\omega(c)}=P(\Gamma,\omega(c))$ over the Coxeter graph. Conversely, the Coxeter elements $c=s_1\cdots s_r$  and $c'=s'_1\cdots s'_r$ are equal as group elements if and only if they are linear extensions of the same poset $P_\omega$. 

Vertices that are sources (respectively, sinks) in $\omega(c)$ are called \emph{initial} (respectively, \emph{terminal}) in $c$, and these are precisely the generators that appear first (resp.\ last) in some reduced expression of $c$. Notice that if $s$ is initial in $c$, then $s$ is terminal in $scs$,
which is a cyclic shift of some reduced expression for $c$, since
\[
s_{x_1}(s_{x_1}s_{x_2}\cdots s_{x_\ell})s_{x_1}=s_{x_2}\cdots s_{x_\ell}s_{x_1}\,.
\]
In other words, conjugating a Coxeter element $c$ by an initial generator $s$ cyclically shifts some reduced expression, and the corresponding acyclic orientations $\omega(c)$ and $\omega(scs)$ differ by converting a source into a sink. This generates an equivalence relation $\equiv$ on the set $\Acyc(\Gamma)$ of acyclic orientations, where we declare two acyclic orientations to be \emph{torically equivalent} if one can be obtained from the other by a sequence of these source-to-sink operations. The name comes from \cite{develin2016toric}, where these equivalence classes were formalized as a cyclic analogue of posets called \emph{toric posets}, via chambers of toric graphic hyperplane arrangements. In \cite{eriksson2009conjugacy}, Eriksson and Eriksson showed that $c,c'\in\C(W)$ are conjugate iff $\omega(c)\equiv\omega(c')$. Said differently, two Coxeter elements are conjugate if and only if one can be transformed into the other via a sequence of cyclic shifts and/or transpositions of commuting generators. Note that one direction of this statement is obvious, but the other direction is highly non-trivial.

In summary, for a fixed Coxeter system $(W,S)$, we have bijections between Coxeter elements and acyclic orientations, and between conjugacy classes and toric equivalence classes. Specifically, these bijections are defined by
\begin{equation}\label{eq:2bijections}
  \def\arraystretch{1.15}
  \begin{array}{cllllc}
    \Acyc(\Gamma)\longto\C(W,S) &&&&&
    \Acyc(\Gamma)/\!\!\equiv\;\longto \Conj(\C(W,S)) \\
     \omega\longmapsto w_1\cdots w_r &&&&& 
    [\omega]\longmapsto\cl_W(w_1\cdots w_r),
  \end{array}
\end{equation}
where $w_1\cdots w_r$ is any linear extension of $P_\omega$ and $\cl_W(w_1\cdots w_r)$ is its conjugacy class in $W$. %These sets are enumerated by evaluations of the Tutte polynomial $T_\Gamma(x,y)$, as defined in \cite{tutte1954contribution}.
% \begin{itemize}
% \item $|\Acyc(\Gamma)|=|\C(W)|=T_\Gamma(2,0)$; see \cite{stanley1973acyclic}.
% \item $|\Acyc(\Gamma)/\!\!\equiv\!|=|\Conj(\C(W))|=T_\Gamma(1,0)$; see \cite{develin2016toric}.
% \end{itemize}
% \simon{It is always tempting to write down all the cool things that one knows, but if we aren't using Tutte polynomials, we shouldn't mention them.}

%We follow the notation of \cite{kleiner2007admissible}
%and \cite{speyer2009powers}, who used these sequences to prove that powers of Coxeter elements in infinite irreducible Coxeter groups are reduced.

In what follows, we will use the notation and terminology of Coxeter groups to talk about quotients of Coxeter groups, i.e.\ groups generated by involutions. That is, when we speak of $(W,S)$, all that is assumed is that the group $W$ is generated by a finite set $S\subset W$ of involutions. The Coxeter graph $\Gamma$ is defined as before, and the edge weights of $\Gamma$ are $m_{i,j}:=|s_is_j|$. Other standard terms such as the set $\C(W)$ of Coxeter elements, reduced expressions, the length of an element, initial and terminal generators, and the acyclic orientation of a Coxeter element, easily and unambiguously carry over. Anytime we are specifically assuming that $W$ is a Coxeter group, we will make this clear. 

Any two Coxeter elements that arise as linear extensions of the same acyclic orientation are clearly equal as elements in $W$. Moreover, two Coxeter elements that are linear extensions of torically equivalent orientations will be conjugate as group elements. This is the ``obvious'' direction of the Erikssons' aforementioned theorem; the converse need not hold for non-Coxeter groups.

In particular, this means that if $W$ is a Coxeter group, and $W'$ a quotient of $W$ (e.g., $W'=W_n$), we have the following commutative diagrams.
\begin{equation}\label{eqn:quotient}
 \xymatrix{\Acyc(\Gamma) \ar[r]^\cong\ar[dr] & \C(W)\ar[d] \\ & \C(W')}
 \hspace{12mm}\xymatrix{\Acyc(\Gamma)/\!\!\equiv \ar[r]^\cong\ar[dr] & \Conj(\C(W))\ar[d] \\ & \Conj(\C(W'))}
\end{equation}
It could happen that two Coxeter elements that are not linear extensions of the same orientation are nevertheless equal in $W'$. Similarly, it could be the case that two Coxeter elements arising from non-torically equivalent orientations could happen to be conjugate for non-Coxeter-theoretic reasons.

Recall that when we are speaking of the toggle group $W_n$, we will assume that a product of toggles, such as $\tau_{i,j}\tau_{k,\ell}$, is performed right-to-left, as in function composition. The following is a direct consequence of the commutative diagrams in Eq.~\eqref{eqn:quotient}.

\begin{prop}\label{prop:columns-and-rows}
The Coxeter element in $W_n$ defined by ``toggling by columns'' (left-to-right, reading each column from top-to-bottom) is the same as the Coxeter element defined by ``toggling by rows'' (top-to-bottom, reading each row from left-to-right). That is, 
\[
\tau_{n-1,n}\tau_{n-2,n}\tau_{n-3,n}\cdots\tau_{1,4}\tau_{2,3}\tau_{1,3}\tau_{1,2}=\tau_{n-1,n}\tau_{n-2,n}\tau_{n-2,n-1}\cdots\tau_{1,5}\tau_{1,4}\tau_{1,3}\tau_{1,2}.
\]
\end{prop}

\begin{proof}
Both of these are linear extensions of the same acyclic orientation of $\Gamma$, namely the one that orients all edges east, and south, and southeast. 
\end{proof}

\begin{lem}
  In the toggle group $W_n$, Coxeter elements have length $\ell=\binom{n}{2}$.
\end{lem}

\begin{proof}
  Let $c$ be a Coxeter element in $W_n$, which is the product of $\binom{n}{2}$ toggles, so $\ell(c)\leq\binom{n}{2}$. To show that equality holds, it suffices to show that each toggle must appear in every reduced expression of $c$. Given any $i<j$, if $P$ contains $(i,j)$, then applying every toggle once to $P$ (in any order) will remove $(i,j)$. Therefore, any expression for $c$ as a product of toggles must contain $\tau_{i,j}$.
\end{proof}

Recall that conjugating a Coxeter element $c$ by an initial generator corresponds to performing a source-to sink operation on the acyclic orientation $\omega(c)$ of $\Gamma$. We define a \emph{$c$-admissible sequence} to be any sequence of generators that arises as a valid sequence of source-to-sink conversions on $\omega(c)$. 
\begin{defn}
Let $W$ be a group generated by a set $S$ of involutions, and $c\in W$ a fixed Coxeter element. A \emph{$c$-admissible sequence} is any sequence of generators $s_{x_1},\dots,s_{x_m}$ such that $s_{x_1}$ is a source of $\omega(c)$, $s_{x_2}$ is a source of $\omega(s_{x_1}cs_{x_1})$, $s_{x_3}$ is a source of $\omega(s_{x_2}s_{x_1}cs_{x_1}s_{x_2})$, and so on. 

Every $c$-admissible sequence defines a canonical group element $a=s_{x_1}\cdots s_{x_m}$ in $W$, and we say that $a^{-1}ca$ is an \emph{admissible conjugation} of $c$.
\end{defn}

%\begin{ex} 
%Conjugation by any initial generator of a Coxeter element is always admissible. 
%\end{ex}

The main theorem of \cite{speyer2009powers} is that if $W$ is an infinite irreducible (that is, $\Gamma$ is connected) Coxeter group, and $s_{x_1},\dots,s_{x_m}$ is a $c$-admissible sequence, then $s_{x_1}\cdots s_{x_m}$ is reduced in $W$. This was the first proof that powers of Coxeter elements are reduced. The utility of $c$-admissible sequences in this paper is that they preserve the number of times a particular arc appears in an orbit. As a result, the homomesies we prove in this paper are preserved under conjugation by a $c$-admissible sequence. 

\section{Homomesy} \label{sec:homomesy}
%%=================================

\begin{defn}\label{defn:homomesic}
Let $X$ be a finite set, $A$ a $\QQ$-vector space (frequently a field such as $\RR$), $f:X\to A$ a function or ``statistic,'' and $T:X\to X$ a bijective function. Then we say that the triple $(X,f,T)$ is \emph{homomesic} if there exists some $c\in A$ so that, for any $T$-orbit $\mcO$, 
\[
\frac{1}{\#\mcO} \sum_{x\in\mcO} f(x)=c.
\] 
We call $c$ the \emph{mean}, and we say that $(X,f,T)$ is $c$-mesic. 
\end{defn}

Even though we use $c$ to denote a Coxeter element, and the mean of a homomesic function, it should always be clear from the context to which we are referring.

\begin{defn} 
The \emph{arc count statistic} $\alpha(P)$ of a noncrossing partition is the number of pairs $(i,j)$ with $1\le i<j\le n$ appearing in $P$. 
\end{defn}

The \emph{block count statistic} $\beta(P)$ can be defined similarly; 
$\beta(P)=|\pi|$ where $\pi$ is the block representation
of the noncrossing partition $P$.
Since a block with $k$ elements contains $k-1$ arcs, it follows that $\alpha(P)+\beta(P)=n$.

We also need to define a certain action on $\NC(n)$.

\begin{defn} An element $w\in W_n$ is called a \emph{partial Coxeter element} if it can be written as $w=\tau_{a_k}\tau_{a_{k-1}}\cdots \tau_{a_1}$, where each $\tau_{a_i}$ is a toggle at some arc, and $a_i\not=a_j$ if $i\not=j$. In other words, each arc appears as a toggle in $w$ at most once, but some might not appear in $w$ at all. \end{defn}

\begin{rem} Much of the theory of Coxeter elements also makes sense for partial Coxeter elements. For instance, a partial Coxeter element $c$ determines an acyclic orientation of the subgraph of $\Gamma_n$ consisting of all edges corresponding to the involutions contained in $c$. We may also talk about admissible conjugations of partial Coxeter elements: a conjugation $s^{-1}cs$ of a partial Coxeter element, for some $s\in W_n$, is admissible if $s^{-1}cs$ is a partial Coxeter element. \end{rem}

The main point of this paper is to understand the distribution of $\alpha$ and its variants in $w$-orbits of $\NC(n)$ for partial Coxeter elements $w\in W_n$. In particular, we show the following:

\begin{thm} \label{thm:arccounthomomesy} Let $w\in W_n$ be any partial Coxeter element that contains every toggle of the form $\tau_{i,i+1}$. Then the triple $(\NC(n),\alpha,w)$ is $\frac{n-1}{2}$-mesic.  This implies also that $(\NC(n),\beta,w)$ is $\frac{n+1}{2}$-mesic. \end{thm}

\begin{ex}
Figure \ref{example} shows the five orbits of $w=\tau_{3,4}\tau_{1,2}\tau_{2,3}\tau_{1,4}$ on $\NC(4)$.  Note that $w$ satisfies the necessary conditions in Theorem \ref{thm:arccounthomomesy} but is not a Coxeter element, since it does not contain $\tau_{1,3}$ or $\tau_{2,4}$.  The figure shows that $(\NC(4),\alpha,w)$ is $\frac{3}{2}$-mesic.

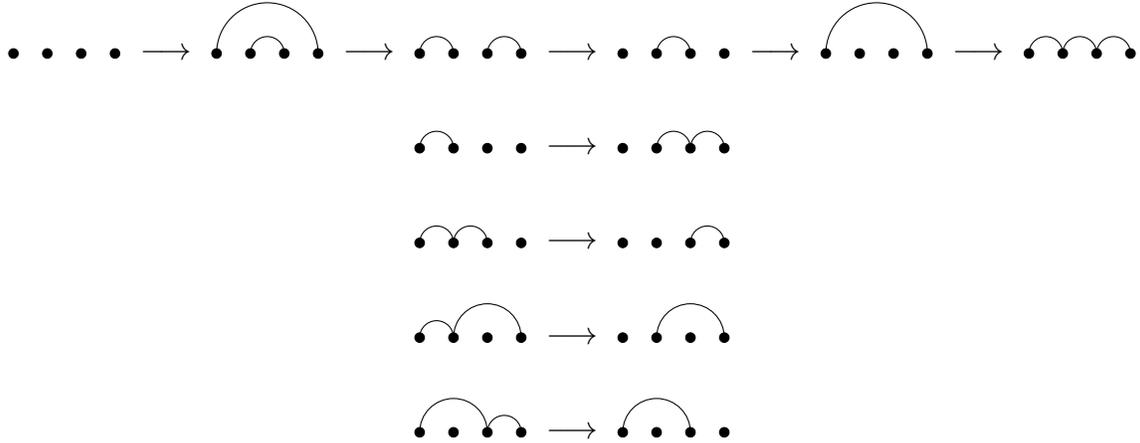
\begin{figure}
\begin{center}
\begin{tikzpicture}[scale=.9]
\filldraw (0,0) circle (2 pt);
\filldraw (0.5,0) circle (2 pt);
\filldraw (1,0) circle (2 pt);
\filldraw (1.5,0) circle (2 pt);

\filldraw (3,0) circle (2 pt);
\filldraw (3.5,0) circle (2 pt);
\filldraw (4,0) circle (2 pt);
\filldraw (4.5,0) circle (2 pt);

\filldraw (6,0) circle (2 pt);
\filldraw (6.5,0) circle (2 pt);
\filldraw (7,0) circle (2 pt);
\filldraw (7.5,0) circle (2 pt);

\filldraw (9,0) circle (2 pt);
\filldraw (9.5,0) circle (2 pt);
\filldraw (10,0) circle (2 pt);
\filldraw (10.5,0) circle (2 pt);

\filldraw (12,0) circle (2 pt);
\filldraw (12.5,0) circle (2 pt);
\filldraw (13,0) circle (2 pt);
\filldraw (13.5,0) circle (2 pt);

\filldraw (15,0) circle (2 pt);
\filldraw (15.5,0) circle (2 pt);
\filldraw (16,0) circle (2 pt);
\filldraw (16.5,0) circle (2 pt);

\draw (3,0) arc (180:0:0.75cm);
\draw (3.5,0) arc (180:0:0.25cm);
\draw (6,0) arc (180:0:0.25cm);
\draw (7,0) arc (180:0:0.25cm);
\draw (9.5,0) arc (180:0:0.25cm);
\draw (12,0) arc (180:0:0.75cm);
\draw (15,0) arc (180:0:0.25cm);
\draw (15.5,0) arc (180:0:0.25cm);
\draw (16,0) arc (180:0:0.25cm);

\node at (2.25, 0) {$\longrightarrow$};
\node at (5.25, 0) {$\longrightarrow$};
\node at (8.25, 0) {$\longrightarrow$};
\node at (11.25, 0) {$\longrightarrow$};
\node at (14.25, 0) {$\longrightarrow$};
\node at (8.25, -1.4) {$\longrightarrow$};
\node at (8.25, -2.8) {$\longrightarrow$};
\node at (8.25, -4.2) {$\longrightarrow$};
\node at (8.25, -5.6) {$\longrightarrow$};

\filldraw (6,-1.4) circle (2 pt);
\filldraw (6.5,-1.4) circle (2 pt);
\filldraw (7,-1.4) circle (2 pt);
\filldraw (7.5,-1.4) circle (2 pt);

\filldraw (9,-1.4) circle (2 pt);
\filldraw (9.5,-1.4) circle (2 pt);
\filldraw (10,-1.4) circle (2 pt);
\filldraw (10.5,-1.4) circle (2 pt);

\filldraw (6,-2.8) circle (2 pt);
\filldraw (6.5,-2.8) circle (2 pt);
\filldraw (7,-2.8) circle (2 pt);
\filldraw (7.5,-2.8) circle (2 pt);

\filldraw (9,-2.8) circle (2 pt);
\filldraw (9.5,-2.8) circle (2 pt);
\filldraw (10,-2.8) circle (2 pt);
\filldraw (10.5,-2.8) circle (2 pt);

\filldraw (6,-4.2) circle (2 pt);
\filldraw (6.5,-4.2) circle (2 pt);
\filldraw (7,-4.2) circle (2 pt);
\filldraw (7.5,-4.2) circle (2 pt);

\filldraw (9,-4.2) circle (2 pt);
\filldraw (9.5,-4.2) circle (2 pt);
\filldraw (10,-4.2) circle (2 pt);
\filldraw (10.5,-4.2) circle (2 pt);

\filldraw (6,-5.6) circle (2 pt);
\filldraw (6.5,-5.6) circle (2 pt);
\filldraw (7,-5.6) circle (2 pt);
\filldraw (7.5,-5.6) circle (2 pt);

\filldraw (9,-5.6) circle (2 pt);
\filldraw (9.5,-5.6) circle (2 pt);
\filldraw (10,-5.6) circle (2 pt);
\filldraw (10.5,-5.6) circle (2 pt);

\draw (6,-1.4) arc (180:0:0.25cm);
\draw (9.5,-1.4) arc (180:0:0.25cm);
\draw (10,-1.4) arc (180:0:0.25cm);
\draw (6,-2.8) arc (180:0:0.25cm);
\draw (6.5,-2.8) arc (180:0:0.25cm);
\draw (10,-2.8) arc (180:0:0.25cm);
\draw (6,-4.2) arc (180:0:0.25cm);
\draw (6.5,-4.2) arc (180:0:0.5cm);
\draw (9.5,-4.2) arc (180:0:0.5cm);
\draw (7,-5.6) arc (180:0:0.25cm);
\draw (6,-5.6) arc (180:0:0.5cm);
\draw (9,-5.6) arc (180:0:0.5cm);
\end{tikzpicture}
\end{center}
\caption{The five orbits of $w=\tau_{3,4}\tau_{1,2}\tau_{2,3}\tau_{1,4}$ on $\NC(4)$.  Notice that in any orbit, the average of the arc count is $\frac{3}{2}$.  Also notice that in general, $\alpha(P)+\alpha(w(P))\not=3$, as is the case for Kreweras complementation (see Section \ref{sec:kreweras}).}
\label{example}
\end{figure}
\end{ex}

The following corollary is a special case of Theorem \ref{thm:arccounthomomesy}.

\begin{cor} If $w\in W_n$ is a Coxeter element, then $(\NC(n),\alpha,w)$ is $\frac{n-1}{2}$-mesic and $(\NC(n),\beta,w)$ is $\frac{n+1}{2}$-mesic. \end{cor}

Hence the arc count statistic is \emph{simultaneously} homomesic for all Coxeter elements and partial Coxeter elements that contain every $\tau_{i,i+1}$. We will also show that there are more refined statistics that are homomesic for certain partial Coxeter elements.

Another consequence of Theorem \ref{thm:arccounthomomesy} is the following.
\begin{cor}\label{cor:eveneven}
Let $n$ be even and $w\in W_n$ be any partial Coxeter element that contains every toggle of the form $\tau_{i,i+1}$.  Then each $w$-orbit of $\NC(n)$ contains an even number of noncrossing partitions. 
\end{cor}

\begin{proof}
The arc count of any noncrossing partition is an integer.  Therefore, the only way for the average arc count across an orbit to be $\frac{n-1}{2}$, which is not an integer for even $n$, is if the orbit contains an even number of noncrossing partitions.
\end{proof}

This gives an example as to how homomesy can be used to prove statements that neither mention homomesy nor the statistic that is homomesic.  There is no other known way to prove Corollary \ref{cor:eveneven}, as there does not appear to be a way to characterize the orbit sizes in general.  For example, in $\NC(6)$, the sizes of the orbits of the Coxeter element $$w=\tau_{4,6}\tau_{3,6}\tau_{2,4}\tau_{1,5}\tau_{2,5}\tau_{1,3}\tau_{3,4}\tau_{1,2}\tau_{1,6}\tau_{2,6}\tau_{3,5}\tau_{2,3}\tau_{1,4}\tau_{5,6}\tau_{4,5}$$ are 4, 22, 46, and 60. There is no noticeable pattern aside from the fact that they are all even.

Figure \ref{example} displays an example of Corollary \ref{cor:eveneven}, as each orbit in the example contains either two or six noncrossing partitions.
%%=================================
\section{Kreweras complementation and the Simion-Ullman involution} \label{sec:kreweras}
%%=================================

The action of the Coxeter element in Proposition~\ref{prop:columns-and-rows} on $\NC(n)$ is actually the inverse of a well-studied action called \emph{Kreweras complementation} introduced in~\cite{kreweras1972sur} and further investigated in~\cite{heitsch2015counting}, defined as follows:

\begin{defn} 
Let $\pi\in \NC(n)$. Draw $\pi$ on a circle, as shown on the right side of Figure~\ref{ncgraphics}, and insert a new point $i'$ immediately clockwise from $i$ along the circle. The Kreweras complement $\kappa(\pi)$ is the coarsest noncrossing partition of the primed numbers in the complement of $\pi$. 
\end{defn}

See Figure~\ref{fig:kreweras} for a pictorial example of Kreweras complementation.

We will now show that Kreweras complementation and the action described in Proposition~\ref{prop:columns-and-rows} are closely related:
\begin{thm}\label{kreweras same as nice}
Let $\pi\in \NC(n)$, and denote by $\kappa(\pi)'$ the partition obtained from $\kappa(\pi)$ by replacing $i$ with $i+1$ for each $1 \leq i \leq n$, such that $n$ is replaced by 1. Then 
\[\kappa(\pi)'=\tau_{n-1,n}\tau_{n-2,n}\tau_{n-2,n-1}\cdots\tau_{2,3}\tau_{1,n}\cdots\tau_{1,5}\tau_{1,4}\tau_{1,3}\tau_{1,2}(\pi).\]
\end{thm}
\begin{proof}
First, note that in order to obtain $\kappa(\pi)'$, we draw $\pi$ on a circle, and insert a new point $i'$ immediately \textbf{counterclockwise} from $i$ along the circle. We then take the coarsest noncrossing partition of the primed numbers in the complement of $\pi$, and this is $\kappa(\pi)'$.

For a $A \subset [n]$, we denote by $\pi(A)$ the restriction of $\pi$ to $A$. That is, $B$ is a block of $\pi(A)$ if and only if $B=A \cap C$ for some block $C$ of $\pi$. For example, if $\pi$ is as in Figure~\ref{nclinear}, then $\pi(\{1,5,7,8,9\})=\{\{1,5\},\{7\},\{8,9\}\}$. In this proof, we denote by $\pi[a:b]$ \mbox{(resp.\ $ \pi({c} \cup [a:b]) $)} the restriction of $\pi$ to the set $\{a,a+1,\ldots,b\}$ (resp.\ $ \{c,a,a+1,\ldots,b\} $) for $c < a,b$. If $a>b$ then we set $\pi[a:b]=\varnothing$ and $ \pi({c} \cup [a:b])=\{\{c\}\}$.

We proceed by induction on $n$. The base cases $n=1,2$ are trivial, so we assume that the claim holds for $n \leq m$, and prove it for $n=m+1$. Let $\pi\in \NC(m+1)$. There are two possible cases: either $\{\{1\}\} \in \pi$ or there exists some $2 \leq j \leq m+1$ such that $(1,j)$ is an arc of $\pi$. In the first case, $\kappa(\pi)'$ is obtained from $\kappa(\pi[2:m+1])'$ by adding the element 1 to the set of $\kappa(\pi[2:m+1])'$ that contains 2 (this is equivalent to adding the arc $(1,2)$ to $\kappa(\pi[2:m+1])'$). On the other hand, the action of $\tau_{1,m+1}\ldots\tau_{1,3}\tau_{1,2}$ on $\pi$ adds the arc $(1,2)$. By the inductive hypothesis and the fact that $\tau_{u,v}$ is not influenced by the arc $(1,2)$ if $2 \leq u,v$, we get
\[\tau_{m,m+1}\tau_{m-1,m+1}\ldots\tau_{2,4}\tau_{2,3}(\pi \cup (1,2))=\kappa(\pi[2:m+1])'\cup (1,2).\]
In conclusion, 
\[\tau_{n-1,n}\tau_{n-2,n}\tau_{n-2,n-1}\cdots\tau_{2,3}\tau_{1,n}\cdots\tau_{1,5}\tau_{1,4}\tau_{1,3}\tau_{1,2}(\pi)=\kappa(\pi)',\] and we are done with the first case.

Consider now the second case. By inspection, 
\[\kappa(\pi)'=\kappa(\pi[2:j-1] \cup \{\{j\}\})' \cup \kappa(\pi({1} \cup [j+1:n]))'.\]
Let us examine the action of $\tau:=\tau_{n-1,n}\tau_{n-2,n}\tau_{n-2,n-1}\cdots\tau_{1,3}\tau_{1,2}$ on $\pi$. First, we would like to show that $\tau(\pi)([2:j])=\kappa(\pi[2:j-1] \cup \{\{j\}\})'$.
Applying $\tau_{1,j}\cdots\tau_{1,3}\tau_{1,2}$ on $\pi$ removes the arc $(1,j)$. Then, the action of $\tau_{1,n}\cdots\tau_{1,j+2}\tau_{1,j+1}$ has no influence on the arcs between the numbers in the set $\{2,3,\ldots,j\}$, and no arc of the form $(u,j)$ for $u<j$ is present. We now apply 
$\tau_{2,j}\cdots\tau_{2,4}\tau_{2,3}$. Note that after this stage, either 2 is connected by an arc to some number $2<b<j$, or 2 is connected by an arc to $j$. If the latter case happens, then by the inductive hypothesis on the partition $\pi[2:j-1] \cup \{\{j\}\}$, we have $\tau(\pi)([2:j])=\kappa(\pi[2:j-1] \cup \{\{j\}\})'$. If the former case happens, let us continue by applying $\tau_{2,n}\cdots\tau_{2,j+2}\tau_{2,j+1}$. This action does nothing, since either $(2,b)$ or $(2,j)$ is present. We now apply $\tau_{3,j}\cdots\tau_{3,5}\tau_{3,4}$, and as before after this stage, either 3 is connected by an arc to some $3<u<j$, or 3 is connected by an arc to $j$. In the latter case, the inductive hypothesis on the partition $\pi[2:j-1] \cup \{\{j\}\}$ leads us again to $\tau(\pi)([2:j])=\kappa(\pi[2:j-1] \cup \{\{j\}\})'$. In the former case we continue with $\tau_{3,n}\cdots\tau_{3,j+2}\tau_{3,j+1}$ (which does nothing) and then with $\tau_{4,j}\cdots\tau_{4,6}\tau_{4,5}$ and so on. The same reasoning as before implies that for any $2 \leq y \leq j-1$, the action $\tau_{y,n}\cdots\tau_{y,j+2}\tau_{y,j+1}$ does nothing, and therefore by the inductive hypothesis $\tau(\pi)([2:j])=\kappa(\pi[2:j-1] \cup \{\{j\}\})'$. Moreover, this reasoning also implies that no arc of the form $(v,w)$ for $v \in \{2,3,\ldots,j-1\}, w \in \{j+1,\ldots,n\}$ is present in $\tau(\pi)$ (and also in any intermediate stage of the action of $\tau$).
Finally, we would like to use the inductive hypothesis on $\pi({1} \cup [j+1:n])$ in order to show that $\tau(\pi)([1,j+1:n])=\kappa(\pi({1} \cup [j+1:n]))'$. In view of our observations earlier, the only thing that we need to show is that $\tau_{j,n}\cdots\tau_{j,j+2}\tau_{j,j+1}\cdots\tau_{2,n}\cdots\tau_{2,4}\tau_{2,3}\tau_{1,n}\cdots\tau_{1,3}\tau_{1,2}(\pi)$ has no arc of the form $(j,t)$ for $j<t$, and then we can apply the inductive hypothesis. Assume in contradiction that the arc $(j,t)$ is present in $\tau_{j,n}\cdots\tau_{j,j+2}\tau_{j,j+1}\cdots\tau_{2,n}\cdots\tau_{2,4}\tau_{2,3}\tau_{1,n}\cdots\tau_{1,3}\tau_{1,2}(\pi)$. This implies that we could add the arc $(1,t)$ in an earlier stage of the process, as part of the action of $\tau_{1,t}\cdots\tau_{1,j}\cdots\tau_{1,3}\tau_{1,2}$ on $\pi$, a contradiction. Therefore, $\tau(\pi)([1,j+1:n])=\kappa(\pi({1} \cup [j+1:n]))'$, and hence \[\kappa(\pi)'=\kappa(\pi[2:j-1] \cup \{\{j\}\})' \cup \kappa(\pi({1} \cup [j+1:n]))'.\]
\end{proof}

%\begin{figure}
%\begin{tikzpicture}
%\draw[dashed] (0,0) circle (3cm);
%\foreach \x in {1,...,8}
%{
%\filldraw (90-45*\x:3) circle (.1cm);
%\path node at (90-45*\x:3.3) {\x};
%\filldraw[color=blue] (67.5-45*\x:3) circle (.1cm);
%\path node at (67.5-45*\x:3.3) {\textcolor{blue}{\x$'$}};
%}
%\draw[pattern = north west lines] (0:3) -- (-90:3) -- (-135:3) -- cycle;
%\draw (180:3) -- (90:3);
%\draw[color=blue] (157.5:3) -- (112.5:3);
%\draw[color=blue,pattern = north west lines,pattern color=blue] (45-22.5:3) -- (180+22.5:3) -- (90-22.5:3) -- cycle;
%\draw[color=blue] (-22.5:3) -- (-45-22.5:3);
%\end{tikzpicture}
%\caption{The noncrossing partition \textcolor{blue}{$\{(1,5),(2,3),(5,8),(6,7)\}$} is the Kreweras complement of the black noncrossing partition $\{(2,4),(4,5),(6,8)\}$.}
%\label{fig:kreweras}
%\end{figure}

\begin{figure}
\begin{center}
\begin{tikzpicture}[scale=0.63]
\begin{scope}[shift={(0,0)}]
\draw[dashed] (0,0) circle (3cm);
\foreach \x in {1,...,8}
{
\filldraw (90-45*\x:3) circle (.1cm);
\path node at (90-45*\x:3.3) {\x};
\filldraw[color=blue] (67.5-45*\x:3) circle (.1cm);
\path node at (67.5-45*\x:3.3) {\textcolor{blue}{\x$'$}};
}
\draw[pattern = north west lines] (0:3) -- (-90:3) -- (-135:3) -- cycle;
\draw (180:3) -- (90:3);
\draw[color=blue] (157.5:3) -- (112.5:3);
\draw[color=blue,pattern = north west lines,pattern color=blue] (45-22.5:3) -- (180+22.5:3) -- (90-22.5:3) -- cycle;
\draw[color=blue] (-22.5:3) -- (-45-22.5:3);
\end{scope}
\begin{scope}[shift={(9,0)}]
\draw[dashed] (0,0) circle (3cm);
\foreach \x in {1,...,8}
{
\filldraw (90-45*\x:3) circle (.1cm);
\path node at (90-45*\x:3.3) {\x};
\filldraw[color=blue] (67.5-45*\x:3) circle (.1cm);
\path node at (67.5-45*\x:3.3) {\textcolor{blue}{\x$'$}};
}
\draw[color=blue, pattern = north west lines, pattern color=blue] (0+22.5:3) -- (-90+22.5:3) -- (-135+22.5:3) -- cycle;
\draw[color=blue] (180+22.5:3) -- (90+22.5:3);
\draw[color=black] (157.5+22.5:3) -- (112.5+22.5:3);
\draw[color=black,pattern = north west lines,pattern color=black] (45-22.5+22.5:3) -- (180+22.5+22.5:3) -- (90-22.5+22.5:3) -- cycle;
\draw[color=black] (-22.5+22.5:3) -- (-45-22.5+22.5:3);
\end{scope}
\begin{scope}[shift={(18,0)}]
\draw[dashed] (0,0) circle (3cm);
\foreach \x in {1,...,8}
{
\filldraw (90-45*\x:3) circle (.1cm);
\path node at (90-45*\x:3.3) {\x};
\filldraw[color=blue] (67.5-45*\x:3) circle (.1cm);
\path node at (67.5-45*\x:3.3) {\textcolor{blue}{\x$'$}};
}
\draw[color=black, pattern = north west lines, pattern color=black] (0+45:3) -- (-90+45:3) -- (-135+45:3) -- cycle;
\draw[color=black] (180+45:3) -- (90+45:3);
\draw[color=blue] (157.5+45:3) -- (112.5+45:3);
\draw[color=blue,pattern = north west lines,pattern color=blue] (45-22.5+45:3) -- (180+22.5+45:3) -- (90-22.5+45:3) -- cycle;
\draw[color=blue] (-22.5+45:3) -- (-45-22.5+45:3);
\end{scope}
\end{tikzpicture}
\end{center}
\caption{Applying the Kreweras complement $\kappa$ to the noncrossing partition $\pi=\{(2,4),(4,5),(6,8)\}$ shown at left yields $\kappa(\pi)=\textcolor{blue}{\{(1,5),(2,3),(5,8),(6,7)\}}$, which is the \textcolor{blue}{blue noncrossing partition} on the left, and the black one in the middle. Applying $\kappa$ twice yields $\kappa^2(\pi)=\{(1,3),(3,4),(5,7)\}$, shown at right. The convex hulls of $\kappa^2(\pi)$ and $\pi$ differ by a counterclockwise rotation of $2\pi/8$ radians.}
\label{fig:kreweras}
\end{figure}
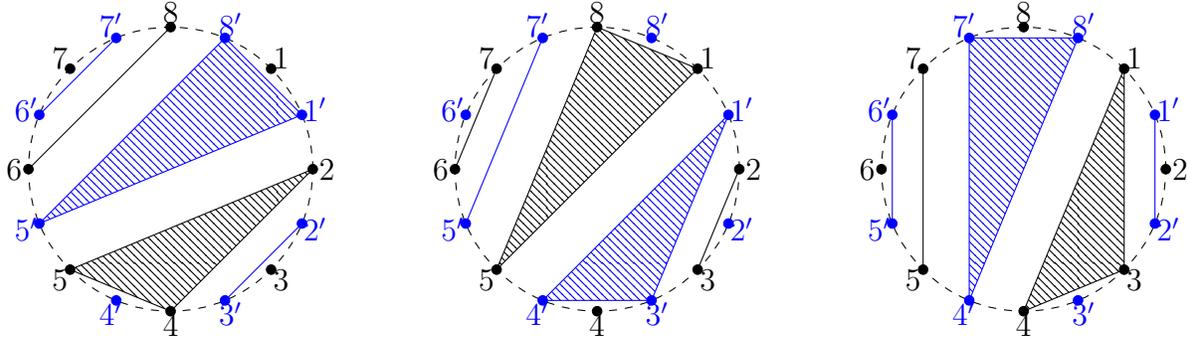

\begin{rem}
Using the notation of Theorem \ref{kreweras same as nice}, $\kappa(\pi)'=\kappa^{-1}(\pi)$. In other words, $$\kappa=\tau_{1,2}\tau_{1,3}\tau_{1,4}\cdots\tau_{1,n}\tau_{2,3}\cdots\tau_{n-2,n-1}\tau_{n-2,n}\tau_{n-1,n}.$$
\end{rem}

\begin{lem}\label{lem:kappa^2}
Let $\pi$ be a noncrossing partition. Applying the Kreweras complement twice to $\pi$ rotates $\pi$ counterclockwise by $2\pi/n$, i.e. $\kappa\circ\kappa$ is a rotation by $2\pi/n$ in the counterclockwise direction in the circular representation of $\NC(n)$, so the order of $\kappa$ divides $2n$. 
\end{lem}

The Simion-Ullman involution is $\lambda=\eta \circ \kappa$, where $\eta$ is the relabeling map
that replaces $i$ by $n-i$ for $1 \leq i < n$
and leaves $n$ fixed.

% \jim{We should show that the Simion-Ullman involution
% fits into our framework; that is, we should show
% that it can be written as a partial Coxeter element
% that contains every adjacent toggle. If this isn't true,
% then perhaps we should remark that our main result applies
% to operations on noncrossing partitions that are conjugate
% to elements of the toggle group via dihedral symmetries
% (and also clarify how much of the dihedral group
% belongs to the toggle group).}

%%=================================
\section{Proof of Theorem~\ref{thm:arccounthomomesy}} \label{sec:mainthmproof}
%%=================================

When searching for homomesies, it helps to define simple indicator function statistics, and then determine which linear combinations of these are homomesic.

\begin{defn} 
We denote the indicator function of the arc $(i,j)$ by $\chi_{i,j}:\NC(n)\to\{0,1\}$ defined as 
\[
\chi_{i,j}(P) = \begin{cases} 1 & (i,j)\in P, \\ 0 & (i,j)\not\in P. \end{cases} 
\] 
\end{defn}

If two Coxeter elements $w,w'\in W_n$ are conjugate, then it follows immediately that there is a bijection of $W_n$ that sends $w$-orbits to $w'$-orbits.

%\michael{I have a problem with the term ``isomorphic'' as used here. If we use it, we should say what is meant, which is sort of what the following does. Isomorphic usually basically means they can be thought of as the same, which is in some ways true as the orbit sizes are the same.  However, the NCP's aren't in the same orbit necessarily.}

\begin{lem} 
Given an admissible conjugation $w'=a^{-1}wa$ of a Coxeter element, there is a natural bijection between $w$-orbits of $\NC(n)$ and $w'$-orbits of $\NC(n)$, given by $\mcO\mapsto a^{-1}\mcO$. 
\[
\xymatrix{\mcO \ar[r]^{\tau_w}\ar[d] & \mcO\ar[d] \\ \!\!a^{-1}\mcO\ar[r]^{\tau_{w'}} & a^{-1}\mcO}
\]
The bijection preserves the size of orbits. 
\end{lem}

There is no reason to expect conjugations of Coxeter elements to preserve statistics or homomesy, because the contents of the orbits are generally scattered. However, in certain cases it surprisingly does.  The following lemma describes how an admissible conjugation preserves any homomesic statistic which is a linear combination of the arc indicator functions.

%\begin{defn} A conjugation $w'=a^{-1}wa$ of a Coxeter element $w$ is said to be \emph{admissible} if $w'$ can be written as a Coxeter element. \end{defn}

\begin{lem} 
Let $w$ be a partial Coxeter element, and let $w'=a^{-1}wa$ be an admissible conjugate of $w$. Let $\mcO$ be a $w$-orbit in $\NC(n)$, and let $\mcO'=a^{-1}\mcO$ be the corresponding orbit of $w'$. Then 
 \[
\sum_{P\in\mcO} \chi_{i,j}(P)=\sum_{P'\in\mcO'}\chi_{i,j}(P').
\] 
\end{lem}

%\simon{This should also work for partial Coxeter elements and admissible conjugates, which we haven't defined yet, but the definition is clear.}

\begin{proof}
Let $w=\tau_{i_1,j_1}\circ\cdots\circ\tau_{i_k,j_k}$ and $\mcO=\{P_1,P_2,\ldots,P_m\}$ such that $w(P_i)=P_{i+1}$ for $1\leq i \leq m-1$, and $w(P_m)=P_1$. Then in order to prove the claim, it is enough to show that it holds for $a=\tau_{i_k,j_k}$. By definition of conjugation, we have \[\mcO'=\{\tau_{i_k,j_k}(P_1),\tau_{i_k,j_k}(P_2),\ldots,\tau_{i_k,j_k}(P_m)\}.\] We will now show that the following holds:
\begin{equation}\label{eq:arcindicator}
\chi_{i,j}(\tau_{i_k,j_k}(P_t)) = \begin{cases} \chi_{i,j}(P_t) & (i,j) \neq (i_k,j_k), \\ \chi_{i,j}(P_{t+1}) & (i,j)=(i_k,j_k), \end{cases} 
\end{equation}
where $1 \leq t \leq m$, and if $t=m$ then we view $t+1$ as 1. Note that Eq.~\eqref{eq:arcindicator} implies the lemma for $a=\tau_{i_k,j_k}$. The case $(i,j) \neq (i_k,j_k)$ follows directly from the fact that applying $\tau_{i_k,j_k}$ may only change the status of the arc $(i_k,j_k)$, and the rest of the arcs stay unchanged. Consider now the case $(i,j)=(i_k,j_k)$, and let $1 \leq t \leq m$. In this case, we have $\chi_{i,j}(P_{t+1})=\chi_{i,j}(w(P_{t}))=\chi_{i,j}(\tau_{i_1,j_1}\circ\cdots\circ\tau_{i_k,j_k}(P_{t}))=\chi_{i,j}(\tau_{i_k,j_k}(P_{t}))$

The last equality follows from the fact that $\tau_{i_k,j_k}$ is the first toggle in the process, and thus the existence (or nonexistence) of the arc $(i_k,j_k)$ in $w(P_{t})$ is determined by this first toggle. Therefore $\chi_{i,j}(P_{t+1})=\chi_{i,j}(\tau_{i_k,j_k}(P_{t}))$, and hence Eq.~\eqref{eq:arcindicator} is proven.
\end{proof}

%The above lemma shows us that two torically equivalent Coxeter elements are conjugate, as conjugating by $\tau_{i_k,j_k}$ is equivalent to changing a source to a sink. \simon{I don't get it: what does the Lemma have to do with toric equivalence? Of course, this is not hard.} By a result of Pretzel \cite{pretzel1986reorienting}, two Coxeter elements are torically equivalent if and only if for a given orientation of the cycles in the Coxeter graph the number of forward edges minus the number of reverse edges in the two induced acyclic orientations are equal. In particular, if the induced Coxeter graph of a subword is a tree, then all Coxeter elements made up of the toggles are conjugate.

\begin{rem} The indicator function $\chi_{i,j}$ is not necessarily homomesic.
For example, consider the action $\tau_{1,3}\circ\tau_{2,3}\circ\tau_{1,2}$ on $\NC(3)$ which forms two orbits --- $\chi_{1,3}$ is $0$ on one orbit and nonzero on the other.  
\end{rem}

%\begin{qn} Which boolean combinations of the indicator functions \emph{are} homomesic? \end{qn}

Before proving Theorem~\ref{thm:arccounthomomesy}, we first define some other statistics.

\begin{defn}\label{psi} Given $k\in[n-1]$ and $P\in \NC(n)$, define the statistic $\psi_k:\NC(n)\to\ZZ$ in the following way:
\begin{eqnarray*}\psi_k(P)&=&2\chi_{k,k+1}(P)+\sum\limits_{1\leq i\leq k-1}\chi_{i,k+1}(P)+\sum\limits_{k+2\leq j\leq n}\chi_{k,j}(P)\\&=&\sum\limits_{1\leq i\leq k}\chi_{i,k+1}(P)+\sum\limits_{k+1\leq j\leq n}\chi_{k,j}(P)\end{eqnarray*} where $\chi_{i,j}$ is the indicator function of the arc $(i,j)$.
\label{def:indicator_statistics}
\end{defn}

Due to the restrictions on arcs with a common left or right endpoint, and arcs that cross, any noncrossing partition can only contain at most one arc that is of the form $(i,k+1)$ or $(k,j)$.  Thus, for any $P\in \NC(n)$, $\psi_k(P)\in\{0,1,2\}$.  Also, $\psi_k(P)$ is fully determined by the following three cases.\begin{itemize}
\item $\psi_k(P)=0$ if and only if $P$ does not contain any arcs of the form $(i,k+1)$ or $(k,j)$.
\item $\psi_k(P)=1$ if and only if $P$ contains an arc of the form $(i,k+1)$ or $(k,j)$ that is not the arc $(k,k+1)$.
\item $\psi_k(P)=2$ if and only if $P$ contains the arc $(k,k+1)$.
\end{itemize}

\begin{thm}\label{psi_homomesy}
Given $k\in[n-1]$, let $T$ either be a Coxeter word, or a partial Coxeter word that contains $\tau_{k,k+1}$. Then the statistic $\psi_k$ is 1-mesic on orbits of $T$.
\label{thm:one_mesic}
\end{thm}

\begin{proof}
To prove that $\psi_k$ is 1-mesic, it is equivalent to prove that in any orbit $\mcO$, \[\#\{P\in\mcO:\psi_k(P)=0\}=\#\{P\in\mcO:\psi_k(P)=2\}.\]

The general strategy is to show that when $\psi_k(P)=0$ and $\ell$ is the smallest positive value such that $\psi_k\left(T^\ell(P)\right)\not=1$, then 
$\psi_k\left(T^\ell(P)\right)=2$ and vice versa.

In other words, we will prove that in any orbit, the number of partitions that do not contain any arcs of the form $(i,k+1)$ or $(k,j)$ is equal to the number of partitions that contain the arc $(k,k+1)$.  Without loss of generality, we assume that in the word $T$, the toggle $\tau_{k,k+1}$ is the toggle that is applied last.  If this is not the case, then we may conjugate $T$ by the toggles that are performed after $\tau_{k,k+1}$, and then the homomesy and orbit sizes will be unchanged.

Let $\{A_1,\ldots,A_m\}$ be the (possibly empty) set of arcs with left endpoint $k$ or right endpoint $k+1$ whose toggles are contained in $T$, excluding $(k,k+1)$.  We will index the $A_i$'s in the order that they are being toggled in $T$.  Note that $T$ may contain other toggles in addition to $\tau_{A_1},\ldots,\tau_{A_m}$ and $\tau_{k,k+1}$.  However, these other arcs do not affect whether or not $(k,k+1)$ can be inserted into a partition, although they may affect whether or not the $A_1,\ldots,A_m$ arcs can be inserted.  The toggles that are significant in this proof are $\tau_{A_1},\ldots,\tau_{A_m}$ and $\tau_{k,k+1}$.

Let $P$ be such that $(k,k+1)\in P$, so $\psi_k(P)=2$.  Then when computing $T(P)$, every toggle $\tau_{A_i}$ will not be able to add the arc $A_i$ because $(k,k+1)$ is in the partition, and then the final toggle $\tau_{k,k+1}$ will remove $(k,k+1)$ from the partition.  Thus, $T(P)$ contains no arcs of the form $(i,k+1)$ or $(k,j)$, so $\psi_k(T(P))=0$.

Let $P$ be such that $\psi_k(P)=0$.  So $P$ contains no arcs of the form $(i,k+1)$ or $(k,j)$.  When computing $T(P)$, the toggle $\tau_{A_1}$ (if $A_1$ exists) will attempt to add the arc $A_1$ to the partition.  It may or not be possible to add that arc, depending on other arcs in the partition.  If $A_1$ cannot be added, then $\tau_{A_2}$ (if $A_2$ exists) will attempt to add the arc $A_2$ to the partition.  Again, that may or may not be possible, and the process continues.  There are two cases that can happen.

Case 1: An arc $A_i$ is added to the partition.  Then the toggles $\tau_{A_{i+1}},\ldots,\tau_{A_m},\tau_{k,k+1}$ will do nothing.  So $T(P)$ contains the arc $A_i$ and thus $\psi_k(T(P))=1$.

Case 2: None of the arcs $A_1,\ldots,A_m$ can be added to the partition when the toggles $\tau_{A_1},\ldots,\tau_{A_m}$ are applied.  Then there are no arcs that interfere with the ability to add the arc $(k,k+1)$, so the final toggle $\tau_{k,k+1}$ adds this arc.  Therefore $\psi_k(T(P))=2$.

Note that if the word $T$ contains no arcs with left endpoint $k$ or right endpoint $k+1$ other than $(k,k+1)$, then $\{A_1,\ldots,A_m\}=\varnothing$, so we go to Case 2 automatically.

Now let $P$ be a partition that contains $A_i$ for some $i$.  When computing $T(P)$, the toggles $\tau_{A_j}$ for $j<i$ do nothing, then $\tau_{A_i}$ removes $A_i$ from the partition.  Then there are two cases for what happens when the toggles $\tau_{A_j}$ for $j>i$ and $\tau_{k,k+1}$ are applied.

Case 1: An arc $A_j$ for some $j>i$ is added to the partition.  Then the toggles $\tau_{A_{j+1}},\ldots,\tau_{A_m},\tau_{k,k+1}$ will do nothing.  So $T(P)$ contains the arc $A_j$ and thus $\psi_k(T(P))=1$.

Case 2: None of the arcs $A_j$ for $j>i$ can be added when those respective toggles are applied.  Then there are no arcs that interfere with the ability to add the arc $(k,k+1)$, so the final toggle $\tau_{k,k+1}$ adds this arc.  Therefore $\psi_k(T(P))=2$.

From this, it is clear that when applying $T$ repeatedly to $P$, the next partition $T^r(P)$ for which $\psi_k\left(T^r(P)\right)\not=1$ satisfies $\psi_k\left(T^r(P)\right)=2$.  Thus, in any orbit $\mcO$, \[\#\{P\in\mcO:\psi_k(P)=0\}=\#\{P\in\mcO:\psi_k(P)=2\},\] so $\psi_k$ is 1-mesic on orbits of $T$.
\end{proof}

The arc count statistic is $\sum_{1\leq i<j\leq n}\chi_{i,j}.$  Given any $i<j$ with $j-i\geq 2$, the coefficient of $\chi_{i,j}$ in $\psi_i$ and $\psi_{j-1}$ is 1, and the coefficient of $\chi_{i,j}$ in all other $\psi_k$ is 0.  For any $i$, the coefficient of $\chi_{i,i+1}$ in $\psi_i$ is 2, and the coefficient of $\chi_{i,i+1}$ in all other $\psi_k$ is 0.  Therefore, the arc count statistic is equal to $\frac{1}{2}\sum_{k=1}^{n-1}\psi_k.$  Theorem~\ref{thm:arccounthomomesy} now follows:

\begin{proof}[Proof of Theorem~\ref{thm:arccounthomomesy}]
By Theorem \ref{psi_homomesy}, $\psi_k$ is 1-mesic on orbits of $T$, for every $k\in[n-1]$.  So the arc count statistic $\alpha=\frac{1}{2}\sum_{k=1}^{n-1}\psi_k$ is $\frac{n-1}{2}$-mesic.
\end{proof}

\section{Toggling independent sets}
%%======================

In this section we generalize our main result to the toggle operations on independent sets. We first introduce some definitions:
 
\begin{defn}
Let $G=(V,E)$ be a simple graph. For $v \in V$, we denote by $N(v)$ the set of neighbors of $v$. A set $W \subset V$ of vertices is called \emph{independent} if no two vertices in $W$ are adjacent. We denote by $\card(W)$ the cardinality of $W$ and $\ind(G)$ the set of all independent sets of $V$.% The \emph{independence number} of $G$ is the size of the maximal independent set of $G$.
\end{defn}

\begin{defn}
Let $G=(V,E)$ be a simple graph and let $v \in V$. The \emph{toggle operation} $\tau_v$ on $\ind(G)$ is defined to be
\[
\tau_v(W)=\begin{cases} W\cup\{v\} & v\not\in W \text{ and } W\cup\{v\}\in\ind(G), \\ W\setminus\{v\} & v\in W, \\ W & \text{otherwise.} \end{cases}
\]
The \emph{toggle group} $W_G$ is the subgroup of the permutation group $S_{\ind(G)}$ generated by the $|V|$ toggle operations.
\label{def:toggle_independent_set}
\end{defn}

\begin{defn} An element $a \in W_G$ is called a \emph{partial Coxeter element} if it can be written as $a=\tau_{v_1}\tau_{v_2}\cdots\tau_{v_k}$, where $v_i \neq v_j$ for $i \neq j$, and $\{v_1,\ldots,v_k\} \subset V$. We call $a$ a \emph{Coxeter element} if $\{v_1,\ldots,v_k\}=V$.
\end{defn}

The toggle operations on $\NC(n)$, defined in the first section, are in fact a special case of Definition~\ref{def:toggle_independent_set}. Define $G=(V,E)$ to be the graph whose vertices represent the arcs in $\NC(n)$, and two vertices are connected by an edge if the corresponding pair of arcs cannot appear together in a noncrossing partition (that is, crossing, left-nesting, or right-nesting). By viewing the elements in $\NC(n)$ as collections of arcs, we see that $W\subset V$ is an independent set if and only if $W \in \NC(n)$. Note that the graph $G$ is just the base graph $\Gamma_n$ defined in the first section and shown in Figure \ref{fig:base-graph}. An example of this is in Figure \ref{fig:NCPasIS} where an element of $\NC(5)$ is displayed as an independent set of $\Gamma_5$. Thus, the action of the group $W_n$ on $\NC(n)$ is isomorphic to the action of the group $W_{\Gamma_n}$ on $\Gamma_n$.

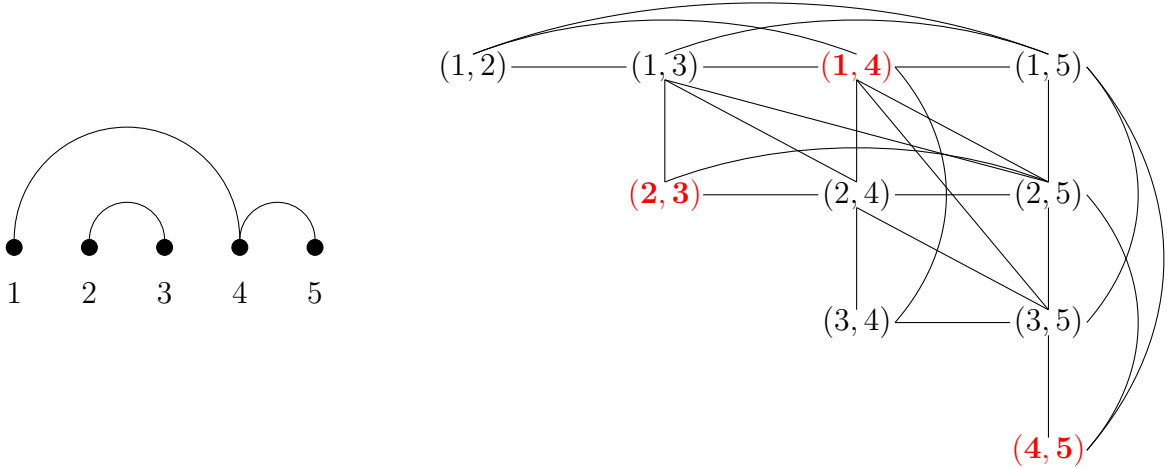
\begin{figure}
\begin{center}
\begin{tikzpicture}
\begin{scope}
\foreach \x in {1,...,5}
{
\filldraw (\x,-1.1) circle (3pt);
\path node at (\x,-1.7) {\x};
}
\draw (1,-1) arc (180:0:1.5cm);
\draw (2,-1) arc (180:0:.5cm);
\draw (4,-1) arc (180:0:.5cm);
\end{scope}
\begin{scope}[shift={(2,3)},xscale=2.55,yscale=1.7]
\draw (2,-0.9) arc (120:60:2cm);
\draw (2,-0.9) arc (120:60:3cm);
\draw (3,-0.9) arc (120:60:2cm);
\draw (3,-1.9) arc (120:60:2cm);
\draw (5.2,-1) arc (30:-30:2cm);
\draw (5.2,-1) arc (30:-30:3cm);
\draw (5.2,-2) arc (30:-30:2cm);
\draw (4.2,-1) arc (30:-30:2cm);
\draw (2.2,-1)--(2.8,-1);
\draw (3.2,-1)--(3.8,-1);
\draw (4.2,-1)--(4.8,-1);
\draw (3.2,-2)--(3.8,-2);
\draw (4.2,-2)--(4.8,-2);
\draw (4.2,-3)--(4.8,-3);
\draw (5,-1.1)--(5,-1.9);
\draw (5,-2.1)--(5,-2.9);
\draw (5,-3.1)--(5,-3.9);
\draw (4,-1.1)--(4,-1.9);
\draw (4,-2.1)--(4,-2.9);
\draw (3,-1.1)--(3,-1.9);
\draw (3,-1.1)--(4,-1.9);
\draw (3,-1.1)--(5,-1.9);
\draw (4,-2.1)--(5,-2.9);
\draw (4,-1.1)--(5,-1.9);
\draw (4,-1.1)--(5,-2.9);
\node at (2,-1) {$(1,2)$};
\node at (3,-1) {$(1,3)$};
\node[red] at (4,-1) {$\bf{(1,4)}$};
\node at (5,-1) {$(1,5)$};
\node[red] at (3,-2) {$\bf{(2,3)}$};
\node at (4,-2) {$(2,4)$};
\node at (5,-2) {$(2,5)$};
\node at (4,-3) {$(3,4)$};
\node at (5,-3) {$(3,5)$};
\node[red] at (5,-4) {$\bf{(4,5)}$};
\end{scope}
\end{tikzpicture}
\end{center}
\caption{The noncrossing partition $\{(1,4),(2,3),(4,5)\}$ of $[5]$ shown at left corresponds to the independent set of $\Gamma_5$ displayed in {\color{red} red} and bold on the right.}
\label{fig:NCPasIS}
\end{figure}

Next, let us introduce an analogue to the statistic from Definition \ref{def:indicator_statistics}:
\begin{defn} Given $G=(V,E)$ and $v\in V$, define the statistic $\psi_v:\ind(G)\to\ZZ$ in the following way:
\begin{eqnarray*}\psi_v(W)&=&2\chi_{v}(W)+\sum\limits_{u \in N(v)}\chi_{u}(W)\end{eqnarray*} where $\chi_{v}$ is the indicator function of the vertex $v$.
\label{def:indicator_statistics_graphs}
\end{defn}
For $G=\Gamma_n$ and $v=(k,k+1)$ this definition coincides with Definition \ref{def:indicator_statistics}.
The following result is a generalization of Theorem~\ref{thm:one_mesic}.
\begin{thm}
Let $G=(V,E)$ be a simple graph. Given $v \in V$, let $T$ be a partial Coxeter element that contains $\tau_v$. If $N(v)$ forms a clique in $G$, then the statistic $\psi_v$ is 1-mesic on orbits of $T$.
\label{thm:one_mesic_independentsets}
\end{thm}
\begin{proof}
Let us examine the proof of Theorem~\ref{thm:one_mesic}. The only property of the graph $\Gamma_n$ that is used in the proof is that for any $k$, the set of neighbors of the vertex $w=(k,k+1)$ is a clique. This implies that at most one of the vertices in $\{w\} \cup N(w)$ can be contained in an independent set in $\Gamma_n$. Thus, similarly to the proof of Theorem~\ref{thm:one_mesic}, we conclude that $\psi_v$ is 1-mesic on orbits of $T$.
\end{proof}
We are ready now to present a generalization of Theorem~\ref{thm:arccounthomomesy}.
\begin{thm}
Let $G=(V,E)$ be a simple graph with maximal independent set $U$ of vertices that satisfies the following two properties:
\begin{enumerate}
\item For any $u \in U$, the set of vertices $N(u)$ forms a clique in $G$.
\item Any vertex in $V \setminus U$ has exactly two neighbors in $U$.
\end{enumerate}
Let $T$ be a partial Coxeter element containing all toggles $\tau_u$ for $u \in U$. Then the triple $(\ind(G),\card,T)$ is $\frac{A}{2}$-mesic, where $A$ is the cardinality of $U$.
\label{thm:main_result_independence}
\end{thm}
\begin{proof}
We have $|U|=A$, so by Theorem~\ref{thm:one_mesic_independentsets}  $\sum\limits_{u \in U}\psi_u$ is $A$-mesic on orbits of $T$. On the other hand, property (2) implies that $\sum\limits_{u \in U}\psi_u=2\sum\limits_{v \in V}\chi_{v}=2\card$. Therefore the triple $(\ind(G),\card,T)$ is $\frac{A}{2}$-mesic.
\end{proof}
\begin{cor}
Let $G$, $A$ and $U$ be as in Theorem~\ref{thm:main_result_independence}, and
let $T$ be a Coxeter element. Then the triple $(\ind(G),\card,T)$ is $\frac{A}{2}$-mesic.
\end{cor}
\begin{defn}
Let $G=(V,E)$ be a graph. Then we say that $G$ is \emph{2-cliquish} if for \emph{some} maximal independent set $U$, the conditions of Theorem~\ref{thm:main_result_independence} are satisfied.
\end{defn}

It is easy to construct various classes of 2-cliquish graphs.  The rest of this section is devoted to describing some such classes.

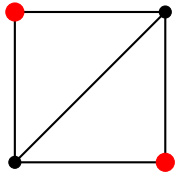
\begin{figure}
\begin{center}
\begin{tikzpicture}
\def\r{0.08}
\def\R{0.12}

\draw[thick, -] (0,0) -- (2,0) -- (2,2) -- (0,2) -- (0,0) -- (2,2);
\draw[fill] (0,0) circle [radius=\r];
\draw[red,fill=red] (2,0) circle [radius=\R];
\draw[fill] (2,2) circle [radius=\r];
\draw[red,fill=red] (0,2) circle [radius=\R];
\end{tikzpicture}
\end{center}
\caption{The complete graph $K_4$ with an edge removed.  The maximal independent set is shown with the large \textcolor{red}{red} vertices.}
\label{k4minusedge}
\end{figure}

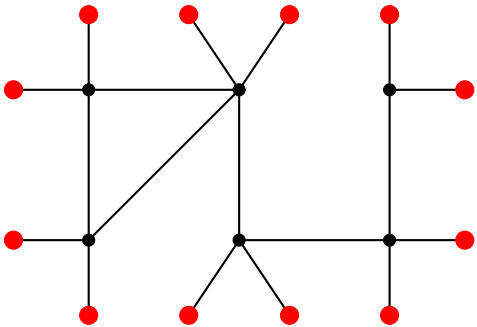
\begin{figure}
\begin{center}
\begin{tikzpicture}
\def\r{0.08}
\def\R{0.12}
\draw[thick, -] (2,2) -- (0,2) -- (0,0) -- (2,2) -- (2,0) -- (4,0) -- (4,2);
\draw[thick, -] (-1,0) -- (0,0) -- (0,-1);
\draw[thick, -] (5,0) -- (4,0) -- (4,-1);
\draw[thick, -] (-1,2) -- (0,2) -- (0,3);
\draw[thick, -] (5,2) -- (4,2) -- (4,3);
\draw[thick, -] (1.33,3) -- (2,2) -- (2.67,3);
\draw[thick, -] (1.33,-1) -- (2,0) -- (2.67,-1);
\draw[fill] (0,0) circle [radius=\r];
\draw[fill] (0,2) circle [radius=\r];
\draw[fill] (4,0) circle [radius=\r];
\draw[fill] (2,0) circle [radius=\r];
\draw[fill] (2,2) circle [radius=\r];
\draw[fill] (4,2) circle [radius=\r];
\draw[red,fill=red] (-1,2) circle [radius=\R];
\draw[red,fill=red] (-1,0) circle [radius=\R];
\draw[red,fill=red] (5,2) circle [radius=\R];
\draw[red,fill=red] (5,0) circle [radius=\R];
\draw[red,fill=red] (0,3) circle [radius=\R];
\draw[red,fill=red] (1.33,3) circle [radius=\R];
\draw[red,fill=red] (2.67,3) circle [radius=\R];
\draw[red,fill=red] (4,3) circle [radius=\R];
\draw[red,fill=red] (0,-1) circle [radius=\R];
\draw[red,fill=red] (1.33,-1) circle [radius=\R];
\draw[red,fill=red] (2.67,-1) circle [radius=\R];
\draw[red,fill=red] (4,-1) circle [radius=\R];
\end{tikzpicture}
\end{center}
\caption{A 2-cliquish graph formed starting with a graph (the subgraph of black vertices) and attaching two new vertices to each vertex.  The large \textcolor{red}{red} vertices form the maximal independent set.}
\label{add2vertices}
\end{figure}

\begin{ex}
Some examples of 2-cliquish graphs are as follows.
\begin{itemize}
\item A complete graph with a single edge removed is 2-cliquish.  The two vertices without an edge connecting them form the maximal independent set.  An example of this type of graph is in Figure \ref{k4minusedge}.
\item Given any graph $G$, define a graph $G'$ in the following way.  Start with $G$.  For every vertex $v$ in the graph, add two new vertices and connect $v$ to the two new vertices.  Then $G'$ is 2-cliquish and the maximal independent set is the added vertices.  An example of this type of graph is in Figure \ref{add2vertices}.
\item Let $C_n$ denote the cycle graph with $n$ vertices.  For every edge $e$ in $C_n$ add a vertex $v_e$ to the graph and add edges from $v_e$ to each endpoint of $e$.  This graph is 2-cliquish.  Its maximal independent set is the set of $n$ added vertices $\{v_e\}$. An example of this type of graph is in Figure \ref{c6plus}.
\end{itemize}
\end{ex}

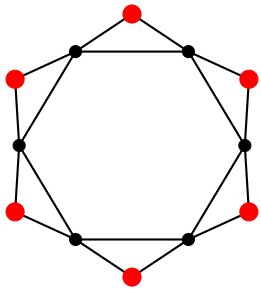
\begin{figure}
\begin{center}
\begin{tikzpicture}
\def\r{0.08}
\def\R{0.12}
\draw[thick, -] (0.75,-1.25) -- (-0.75,-1.25) -- (-1.5,0) -- (-0.75,1.25) -- (0.75,1.25) -- (1.5,0) -- (0.75,-1.25);
\draw[thick, -] (0.75,-1.25) -- (0,-1.75) -- (-0.75,-1.25);
\draw[thick, -] (0.75,1.25) -- (0,1.75) -- (-0.75,1.25);
\draw[thick, -] (-0.75,-1.25) -- (-1.554,-0.882) -- (-1.5,0);
\draw[thick, -] (-0.75,1.25) -- (-1.554,0.882) -- (-1.5,0);
\draw[thick, -] (0.75,-1.25) -- (1.554,-0.882) -- (1.5,0);
\draw[thick, -] (0.75,1.25) -- (1.554,0.882) -- (1.5,0);
\draw[fill] (0.75,-1.25) circle [radius=\r];
\draw[fill] (-0.75,-1.25) circle [radius=\r];
\draw[fill] (-1.5,0) circle [radius=\r];
\draw[fill] (-0.75,1.25) circle [radius=\r];
\draw[fill] (0.75,1.25) circle [radius=\r];
\draw[fill] (1.5,0) circle [radius=\r];
\draw[red,fill=red] (0,1.75) circle [radius=\R];
\draw[red,fill=red] (0,-1.75) circle [radius=\R];
\draw[red,fill=red] (-1.554,-0.882) circle [radius=\R];
\draw[red,fill=red] (1.554,-0.882) circle [radius=\R];
\draw[red,fill=red] (-1.554,0.882) circle [radius=\R];
\draw[red,fill=red] (1.554,0.882) circle [radius=\R];
\end{tikzpicture}
 \end{center}
\caption{The graph formed from $C_6$ with an extra vertex added for each edge adjacent to the endpoints of the corresponding edge.  This graph is 2-cliquish and the large \textcolor{red}{red} vertices form the maximal independent set.}
\label{c6plus}
\end{figure}

The following theorem describes ways to form 2-cliquish graphs from other 2-cliquish graphs.

\begin{thm}\label{thm:formnewgraphs}
Let $G=(V,E)$ and $G'=(V',E')$ be 2-cliquish graphs with maximal independent sets $U$ and $U'$ respectively.
\begin{enumerate}
\item The disjoint union of $G$ and $G'$ is 2-cliquish with maximal independent set $U\cup U'$.
\item Let $e$ be an edge not in $E$ with endpoints in $V\setminus U$.  Then $(V,E\cup\{e\})$ is 2-cliquish with maximal independent set $U$.
\item Let $e$ be an edge in $E$ with endpoints $v,w\in V\setminus U$ such that $v$ and $w$ do not have a common neighbor in $U$.  Then $(V,E\setminus\{e\})$ is 2-cliquish with maximal independent set $U$.
\end{enumerate}
\end{thm}

\begin{proof}
Part (1) is clear from the definition.  For part (2), consider two vertices $v,w\in V\setminus U$ that are not adjacent in $G$.  Since $v$ and $w$ are not adjacent, they cannot have a common neighbor in $U$, since $N(u)$ is a clique for all $u\in U$.  Thus, adding an edge $e$ between $v$ and $w$ does not change the fact that $N(u)$ is a clique for all $u\in U$.  Since the endpoints of $e$ are in $V\setminus U$, the graph $(V,E\cup\{e\})$ also satisfies the condition that every vertex in $V\setminus U$ has exactly two neighbors in $U$.  It also does not change the fact that $U$ is a maximal independent set.  Therefore, $(V,E\cup\{e\})$ is 2-cliquish.

To prove (3), let $e$ be an edge in $E$ with endpoints $v,w\in V\setminus U$ such that $v$ and $w$ do not have a common neighbor in $U$.  Since $e$ has no endpoints in $U$, $(V,E\setminus\{e\})$ satisfies the condition that every vertex in $V\setminus U$ has exactly two neighbors in $U$.  Also, $N(u)$ is a clique for all $u\in U$ because this is true for $G$ so it is true when removing an edge between vertices without a common neighbor in $U$.  Thus, $(V,E\cup\{e\})$ is 2-cliquish.
\end{proof}

We now discuss how to generate all 2-cliquish graphs with a given number of vertices.

\begin{defn}
Let $G=(V,E)$ be 2-cliquish.
\begin{itemize}
\item We say $G$ is \emph{skeletal} if no edges can be removed from it as in part (3) of Theorem \ref{thm:formnewgraphs}.
\item The graph formed from removing edges from $G$ when possible in accordance with part (3) of Theorem \ref{thm:formnewgraphs} is said to be the \emph{skeletalization} of $G$.
\end{itemize}
\end{defn}

In order to generate 2-cliquish graphs, it suffices to begin with the skeletal graphs, and add edges when possible as in part (2) of Theorem \ref{thm:formnewgraphs}.  Figure \ref{fig:skeletal} shows an example of this.  A skeletal graph $G$ is on the left.  There are two pairs of elements that can be connected by edges as in part (2) of Theorem \ref{thm:formnewgraphs}.  This leads to the four 2-cliquish graphs whose skeletalization is $G$.

\begin{figure}
\begin{center}
\begin{tikzpicture}
\def\r{0.08}
\def\R{0.12}
\draw[thick, -] (0,1) -- (1,2) -- (0,3) -- (-1,2) -- (0,1);
\draw[thick, -] (1,2) --(-1,2);
\draw[thick, -] (-1,0) -- (0,0) -- (1,0);
\draw[red,fill=red] (0,1) circle [radius=\R];
\draw[fill] (1,2) circle [radius=\r];
\draw[red,fill=red] (0,3) circle [radius=\R];
\draw[fill] (-1,2) circle [radius=\r];
\draw[red,fill=red] (-1,0) circle [radius=\R];
\draw[fill] (0,0) circle [radius=\r];
\draw[red,fill=red] (1,0) circle [radius=\R];

\draw[thick, -] (4,1) -- (5,2) -- (4,3) -- (3,2) -- (4,1);
\draw[thick, -] (5,2) --(3,2);
\draw[thick, -] (3,0) -- (4,0) -- (5,0);
\draw[thick, -] (4,0) -- (3,2);
\draw[red,fill=red] (4,1) circle [radius=\R];
\draw[fill] (5,2) circle [radius=\r];
\draw[red,fill=red] (4,3) circle [radius=\R];
\draw[fill] (3,2) circle [radius=\r];
\draw[red,fill=red] (3,0) circle [radius=\R];
\draw[fill] (4,0) circle [radius=\r];
\draw[red,fill=red] (5,0) circle [radius=\R];

\draw[thick, -] (8,1) -- (9,2) -- (8,3) -- (7,2) -- (8,1);
\draw[thick, -] (9,2) --(7,2);
\draw[thick, -] (7,0) -- (8,0) -- (9,0);
\draw[thick, -] (8,0) -- (9,2);
\draw[red,fill=red] (8,1) circle [radius=\R];
\draw[fill] (9,2) circle [radius=\r];
\draw[red,fill=red] (8,3) circle [radius=\R];
\draw[fill] (7,2) circle [radius=\r];
\draw[red,fill=red] (7,0) circle [radius=\R];
\draw[fill] (8,0) circle [radius=\r];
\draw[red,fill=red] (9,0) circle [radius=\R];

\draw[thick, -] (12,1) -- (13,2) -- (12,3) -- (11,2) -- (12,1);
\draw[thick, -] (13,2) --(11,2);
\draw[thick, -] (11,0) -- (12,0) -- (13,0);
\draw[thick, -] (12,0) -- (11,2);
\draw[thick, -] (12,0) -- (13,2);
\draw[red,fill=red] (12,1) circle [radius=\R];
\draw[fill] (13,2) circle [radius=\r];
\draw[red,fill=red] (12,3) circle [radius=\R];
\draw[fill] (11,2) circle [radius=\r];
\draw[red,fill=red] (11,0) circle [radius=\R];
\draw[fill] (12,0) circle [radius=\r];
\draw[red,fill=red] (13,0) circle [radius=\R];
\end{tikzpicture}
\end{center}
\caption{The four 2-cliquish graphs whose skeletalization is the graph on the left.  The two in the middle are isomorphic, so they are considered the same unlabeled graph.}
\label{fig:skeletal}
\end{figure}
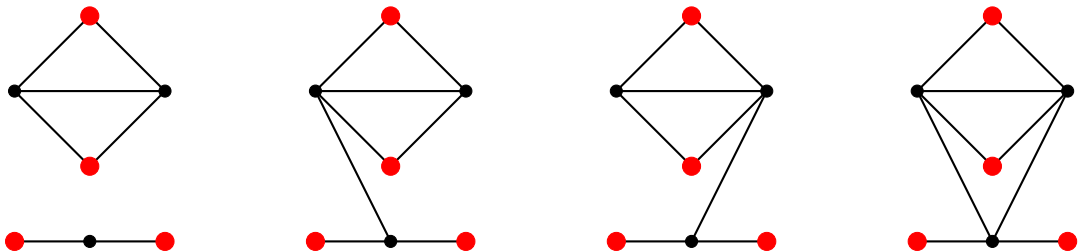

A \emph{multigraph} is a graph that may contain multiple edges with the same pair of endpoints.

\begin{thm}\label{thm:skeletal bijection}
There is a bijection between pairs $(\Gamma,U)$, where $\Gamma$ is a skeletal 2-cliquish graph with $n$ vertices and $U$ is a maximal independent set of $\Gamma$, and loopless multigraphs $G=(V,E)$ that satisfy $|V|+|E|=n$.
\end{thm}

\begin{proof}
Let $\Gamma=(V,E)$ be a skeletal 2-cliquish graph with maximal independent set $U$, and consider the following construct of a loopless multigraph $(V',E')$ from $\Gamma$:  Let $V'=U$, and for each vertex $v\in V\setminus U$,  construct an edge $e'\in E'$, as follows: $v$ has exactly two neighbors in $U$, say $u_1$ and $u_2$; let $e'$ be an edge connecting $u_1$ and $u_2$. Thus $|V'|+|E'|=|U|+|V\setminus U|=|V|=n$, and it is clear that $(V',E')$ is loopless.

For the reverse direction, let $(V',E')$ be a loopless multigraph with $|V'|+|E'|=n$. Construct a graph $\Gamma=(V,E)$ whose vertices are in bijection with those of $V'$ and $E'$; let $U=V'$ and $V\setminus U=E'$. Connect $v\in V\setminus U$ and $u\in U$ with an edge if $u$ is an endpoint of $v$ in $V'$, and connect $v_1,v_2\in V\setminus U$ with an edge if they share a common endpoint in $V'$. (We never connect two elements of $U$ with an edge.) It is clear that $(V',E')$ is 2-cliquish with maximal independent set $U$, and that this map is the inverse of the one in the other direction.
\end{proof}

\begin{ex}
An example of this bijection can be seen in Figure \ref{bijection example}.  We start with the multigraph $M$ on the left and construct the skeletal 2-cliquish graph on the right.  The vertices $A,B,C,D,E$ of the multigraph correspond to the vertices $a,b,c,d,e$ in the skeletal graph.  The set $\{a,b,c,d,e\}$ is the maximal independent set of our skeletal graph.  The edges $e_1,e_2,e_3,e_4$ correspond to the vertices $v_1,v_2,v_3,v_4$ of the skeletal graph.  We use the multigraph to determine which two vertices in $\{a,b,c,d,e\}$ the other vertices are adjacent to.  For example, the edge $e_1$ has endpoints $A$ and $B$, so we add edges from $v_1$ to $a$ and $b$.  Lastly, whenever two vertices have a common neighbor in the independent set $\{a,b,c,d,e\}$, we must add an edge connecting them.  Therefore, we place an edge between $v_1$ and $v_2$, another between $v_1$ and $v_3$, another between $v_2$ and $v_3$, and another between $v_3$ and $v_4$.
\end{ex}

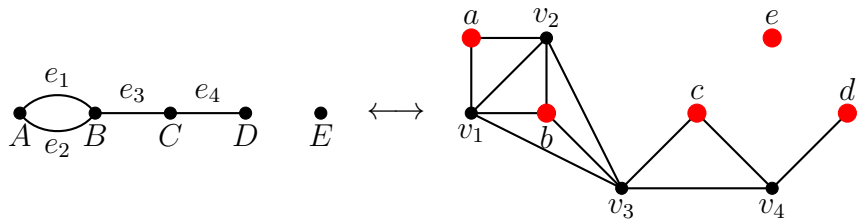
\begin{figure}
\begin{center}
\begin{tikzpicture}
\def\r{0.08}
\def\R{0.12}
\draw[thick, -] (1,0) -- (2,0) -- (3,0);
\draw[thick, -] (0,0) to[out=55,in=125] (1,0);
\draw[thick, -] (0,0) to[out=-55,in=-125] (1,0);
\draw[fill] (0,0) circle [radius=\r];
\draw[fill] (1,0) circle [radius=\r];
\draw[fill] (2,0) circle [radius=\r];
\draw[fill] (3,0) circle [radius=\r];
\draw[fill] (4,0) circle [radius=\r];
\node[below] at (0,0) {$A$};
\node[below] at (1,0) {$B$};
\node[below] at (2,0) {$C$};
\node[below] at (3,0) {$D$};
\node[below] at (4,0) {$E$};
\node[above] at (0.5,0.2) {$e_1$};
\node[below] at (0.5,-0.2) {$e_2$};
\node[above] at (1.5,0) {$e_3$};
\node[above] at (2.5,0) {$e_4$};
\node at (5,0) {$\longleftrightarrow$};
\draw[thick, -] (6,0) -- (6,1) -- (7,1) -- (7,0) -- (6,0) -- (7,1) -- (8,-1) -- (9,0) -- (10,-1) -- (8,-1) -- (6,0);
\draw[thick, -] (7,0) -- (8,-1);
\draw[thick, -] (11,0) -- (10,-1);
\draw[red,fill=red] (6,1) circle [radius=\R];
\draw[fill] (6,0) circle [radius=\r];
\draw[fill] (7,1) circle [radius=\r];
\draw[red,fill=red] (7,0) circle [radius=\R];
\draw[fill] (8,-1) circle [radius=\r];
\draw[red,fill=red] (9,0) circle [radius=\R];
\draw[fill] (10,-1) circle [radius=\r];
\draw[red,fill=red] (10,1) circle [radius=\R];
\draw[red,fill=red] (11,0) circle [radius=\R];
\node[yshift=8] at (6,1) {$a$};
\node[below] at (6,0) {$v_1$};
\node[above] at (7,1) {$v_2$};
\node[yshift=-9] at (7,0) {$b$};
\node[below] at (8,-1) {$v_3$};
\node[yshift=8] at (9,0) {$c$};
\node[below] at (10,-1) {$v_4$};
\node[yshift=8] at (10,1) {$e$};
\node[yshift=9] at (11,0) {$d$};
\end{tikzpicture}
\end{center}
\caption{An example showing the bijection in Theorem \ref{thm:skeletal bijection}.}
\label{bijection example}
\end{figure}

The following corollary is clear using the bijection constructed in the proof of Theorem \ref{thm:skeletal bijection}.

\begin{cor} A pair $(\Gamma,U)$ refers to a skeletal 2-cliquish graph $\Gamma$ with maximal independent set $U$ as in Theorem \ref{thm:skeletal bijection}.
\begin{enumerate}
\item There is a bijection between pairs $(\Gamma,U)$ such that $|\Gamma|=n$ and $\Gamma$ has no isolated vertices, and loopless multigraphs $G=(V,E)$ with no isolated vertices that satisfy $|V|+|E|=n$.
\item There is a bijection between pairs $(\Gamma,U)$ with $\Gamma$ connected and $|\Gamma|=n$, and connected loopless multigraphs $G=(V,E)$ that satisfy $|V|+|E|=n$.
\item There is a bijection between pairs $(\Gamma,U)$ satisfying $|\Gamma|=n$ and $|U|=A$, and loopless multigraphs $G=(V,E)$ that satisfy $|V|=A$ and $|E|=n-A$.
\end{enumerate}
\end{cor}

Note that if one is interested in generating 2-cliquish graphs without isolated vertices, it is enough to start with skeletal 2-cliquish graphs without isolated vertices.  However, if one is interested in generating connected 2-cliquish graphs, it is not enough to begin with connected skeletal 2-cliquish graphs, as the skeletalization of a connected graph can be disconnected, as can be seen in Figure \ref{fig:skeletal}.

\bibliography{homomesy}
\bibliographystyle{halpha}

\iffalse

\fi

\end{document}